\documentclass[12pt]{amsart}

\usepackage{amssymb, amsmath, amsthm, mathrsfs, braket, xspace}
\usepackage[margin=1.0in]{geometry}
\numberwithin{equation}{section}
\usepackage[colorlinks=true]{hyperref}

\usepackage{aliascnt}

\usepackage{mathtools}
\mathtoolsset{showonlyrefs=true} 

\usepackage[all,2cell]{xy}
\objectmargin+{1mm}
\labelmargin+{0.8mm}
\SelectTips{cm}{12}

\usepackage{enumitem}
\setlist[enumerate]{itemsep=0pt,label=$(\mathrm{\roman*})$, topsep=5pt}
\setlist[itemize]{itemsep=0pt, topsep=5pt, labelindent=\parindent,leftmargin=*}
\setlist[description]{itemsep=0pt, topsep=5pt, leftmargin=*}

\newtheorem{thm}{Theorem}[section]

\newaliascnt{cor}{thm}
\newtheorem{cor}[cor]{Corollary}
\aliascntresetthe{cor}

\newaliascnt{lem}{thm}
\newtheorem{lem}[lem]{Lemma}
\aliascntresetthe{lem}

\newaliascnt{prop}{thm}
\newtheorem{prop}[prop]{Proposition}
\aliascntresetthe{prop}

\theoremstyle{definition}
\newaliascnt{dfn}{thm}
\newtheorem{dfn}[dfn]{Definition}
\aliascntresetthe{dfn}

\newaliascnt{rem}{thm}
\newtheorem{rem}[rem]{Remark}
\aliascntresetthe{rem}

\newaliascnt{ex}{thm}
\newtheorem{ex}[ex]{Example}
\aliascntresetthe{ex}

\newtheorem{claim}{Claim}

\newtheorem*{oneclaim}{Claim}

\newcommand{\ab}{\mathrm{ab}}
\newcommand{\A}{\mathscr{A}}

\newcommand{\Ahat}{\widehat{A}}
\newcommand{\Abar}{\ol{A}}
\DeclareMathOperator{\Alb}{Alb}
\DeclareMathOperator{\Aut}{Aut}

\newcommand{\Cor}{\operatorname{Cor}}
\newcommand{\Cf}{\textit{cf.}\;}
\DeclareMathOperator{\Coker}{Coker}
\DeclareMathOperator{\ck}{Coker}

\renewcommand{\div}{\mathrm{div}}
\DeclareMathOperator{\CH}{CH_0}

\renewcommand{\d}{\partial}
\newcommand{\dX}{\d_X}
\newcommand{\dJ}{\d_J}
\newcommand{\dA}{\d_A}
\newcommand{\dE}{\d_E}

\newcommand{\ds}{\displaystyle}

\newcommand{\E}{\mathscr{E}}
\newcommand{\Ehat}{\widehat{E}}

\newcommand{\Ebar}{\ol{E}}
\newcommand{\et}{\mathrm{et}}
\DeclareMathOperator{\End}{End}

\newcommand{\F}{\mathbb{F}}

\newcommand{\Fk}{\F_k}
\newcommand{\FK}{\F_K}
\newcommand{\fin}{\mathrm{fin}}

\newcommand{\geo}{\mathrm{geo}}

\DeclareMathOperator{\Gal}{Gal}
\newcommand{\Gm}{\mathbb{G}_{m}}

\DeclareMathOperator{\Hom}{Hom}
\renewcommand{\H}{\mathscr{H}}

\newcommand{\isomto}{\xrightarrow{\simeq}}
\renewcommand{\Im}{\operatorname{Im}}
\newcommand{\img}{\operatorname{Im}}
\newcommand{\inj}{\hookrightarrow}
\newcommand{\ilim}{\varinjlim}
\newcommand{\Id}{\operatorname{id}}
\newcommand{\Ie}{\emph{i.e.},\ }
\DeclareMathOperator{\Jac}{Jac}
\newcommand{\Jbar}{\ol{J}}
\newcommand{\Jhat}{\wh{J}}
\newcommand{\J}{\mathscr{J}}

\DeclareMathOperator{\Ker}{Ker}
\renewcommand{\ker}{\operatorname{Ker}}
\newcommand{\Kt}{K^{\times}}
\newcommand{\kt}{k^{\times}}

\newcommand{\kbar}{\ol{k}}
\newcommand{\kur}{{k^{\ur}}}

\newcommand{\kX}{k(X)}
\newcommand{\kV}{k(V)}

\newcommand{\m}{\mathfrak{m}}

\newcommand{\mK}{\m_K}
\newcommand{\Mur}{M^{\ur}}
\newcommand{\M}{\mathscr{M}}

\newcommand{\N}{\mathscr{N}}
\newcommand{\Nhat}{\widehat{N}}

\newcommand{\ol}[1]{\overline{#1}}
\renewcommand{\O}{\mathcal{O}}
\newcommand{\OK}{\O_K}
\newcommand{\Ok}{\O_k}
\newcommand{\OKt}{\OK^{\times}}

\newcommand{\onto}[1]{\stackrel{#1}{\to}}
\newcommand{\otimesZ}{\otimes_{\Z}}
\newcommand{\otimesM}{\otimes}

\newcommand{\PF}{$(\mathbf{PF})\,$}
\newcommand{\piab}{\pi_1^{\ab}}
\newcommand{\piabX}{\piab(X)}
\newcommand{\piabV}{\piab(V)}
\newcommand{\piabXbar}{\piab(\Xbar)}
\newcommand{\piabVbar}{\piab(\Vbar)}

\newcommand{\piabXgeo}{\piabX^{\geo}}
\newcommand{\piabVgeo}{\piabV^{\geo}}

\newcommand{\piabXbargeo}{\piab(\Xbar)^{\geo}}

\newcommand{\piabVram}{\piabV_{\ram}}

\newcommand{\piabXgeoram}{\piabXgeo_{\ram}}
\newcommand{\piabVgeoram}{\piabVgeo_{\ram}}

\newcommand{\plim}{\varprojlim}

\newcommand{\Q}{\mathbb{Q}}
\newcommand{\Qp}{\Q_p}
\newcommand{\QZ}{\Q/\Z}

\newcommand{\Res}{\operatorname{Res}}
\newcommand{\ram}{\mathrm{ram}}

\newcommand{\sn}{\smallskip\noindent}
\renewcommand{\sp}{\operatorname{sp}}
 
\newcommand{\surj}{\twoheadrightarrow}
\newcommand{\ssm}{\smallsetminus}
\newcommand{\Spec}{\operatorname{Spec}}

\newcommand{\SKX}{SK_1(X)}

\newcommand{\tor}{\mathrm{tor}}
\newcommand{\Tor}{\operatorname{Tor}}

\newcommand{\Ubar}{\ol{U}}
\newcommand{\ur}{\mathrm{ur}}
\newcommand{\urab}{\ur,\ab}
\newcommand{\V}{\mathscr{V}}
\newcommand{\Vbar}{\ol{V}}
\newcommand{\VX}{V(X)}
\newcommand{\VXdiv}{\VX_{\div}}
\newcommand{\VXfin}{\VX_{\fin}}
\newcommand{\VE}{V(E)}

\newcommand{\W}{\mathscr{W}}
\newcommand{\Wbar}{\ol{W}}
\newcommand{\wt}[1]{\widetilde{#1}}
\newcommand{\wh}[1]{\widehat{#1}}

\newcommand{\X}{\mathscr{X}}
\newcommand{\Xbar}{\ol{X}}
\newcommand{\xbar}{\ol{x}}

\newcommand{\Z}{\mathbb{Z}}

\newcommand{\Zp}{\Z_p}

\title[Abelian geometric fundamental groups]{Abelian geometric fundamental groups for curves over a $p$-adic field}

\author[E. Gazaki]{Evangelia Gazaki} \address[E. Gazaki]{Department of Mathematics, University of Virginia, 221 Kerchof Hall, 141 Cabell Dr., Charlottesville, VA, 22904, USA}
\email{eg4va@virginia.edu}
\author[T. Hiranouchi]{Toshiro Hiranouchi}\address[T. Hiranouchi]{Department of Basic Sciences, Graduate School of Engineering, 
Kyushu Institute of Technology, 
1-1 Sensui-cho, Tobata-ku, Kitakyushu-shi, 
Fukuoka 804-8550 JAPAN}
\email{hira@mns.kyutech.ac.jp}

\begin{document}
\pagenumbering{arabic}
\maketitle

\begin{abstract}
For a curve $X$ over a $p$-adic field $k$, 
using the class field theory of $X$ due to 
 S.~Bloch and S.~Saito 
we study 
the abelian 
geometric fundamental group $\piabXgeo$ of $X$. 
In particular, it is investigated a subgroup of $\piabXgeo$ 
which classifies 
the geometric and abelian coverings of $X$ 
which allow possible ramification over the special fiber 
of the model of $X$.  
Under the assumptions that $X$ has a $k$-rational point,  $X$ has good reduction and
 its Jacobian variety has good ordinary reduction,   
we give 
some upper and lower bounds 
of this subgroup of $\piabXgeo$. 
\end{abstract}


\section{Introduction}
Let $k$ be a \textbf{$p$-adic field}, that is, a finite extension of $\Qp$, with residue field $\Fk$. 
In this note, we investigate the 
abelian fundamental group $\piabX$ 
for a projective smooth and geometrically connected
curve $X$ over $k$.  
The structure map $X\to \Spec(k)$ induces 
the short exact sequence 
\begin{equation}
	\label{seq:geo}
	0 \to \piabXgeo \to \piabX \to G_k^{\ab} = \piab(\Spec(k)) \to 0, 
\end{equation}
where $\piabXgeo$ is defined by the exactness 
and here we call this the \textbf{geometric fundamental group} of $X$.  
Local class field theory describes $G_k^{\ab}$ sufficiently 
so that our interest is in $\piabXgeo$. 
Now, we restrict our attention to 
the case where 
$X$ has \textbf{good reduction} 
in the sense 
that the special fiber $\Xbar := \X\otimes_{\Ok}\Fk$ 
of the regular model $\X$ over $\Ok$ of $X$ 
is a smooth curve over $\Fk$ 
and also 
$X$ has a $k$-rational point. 
The short exact sequence \eqref{seq:geo} splits. 
There is a map called the \textbf{specialization map}  
$\piab(X)\xrightarrow{\sp}\piab(\Xbar)$ (\Cf \eqref{eq:sp}) 
and this induces 
\begin{equation}
	\label{seq:ram}
	0 \to \piabXgeoram \to \piabXgeo \xrightarrow{\sp} \piabXbargeo \to 0, 
\end{equation}
where $\piabXbargeo := \Ker(\piabXbar\to G_{\Fk})$ is 
the (abelian) geometric fundamental group of $\Xbar$ and 
$\piabXgeoram$ is defined by the exactness again.  
The fundamental group $\piabXgeoram$ classifies 
the geometric (abelian) coverings of $X$ which are \emph{completely ramified} 
over the special fiber $\Xbar$ (for the precise description and definition, see \autoref{sec:Galois}). 
The classical class field theory (for the curve $\Xbar$ over the finite field $\Fk$) 
says that the reciprocity map induces an isomorphism 
$\rho_{\Xbar} \colon \Jbar \isomto \piabXbar^{\geo}$, 
where $\Jbar = \Jac(\Xbar)$ is the Jacobian variety of $\Xbar$. 
Our main result describes the structure of 
the remaining part $\piabXgeoram$ 
by using an invariant concerning the Jacobian variety $J = \Jac(X)$ of $X$. 

\begin{thm}[\Cf \autoref{cor:main}]
\label{thm:main_intro}
	Let $X$ be a projective smooth curve over $k$ with $X(k)\neq \emptyset$, 
	and $J = \Jac(X)$ the Jacobian variety of $X$. 
	We assume that $X$ has good reduction, 
	and the Jacobian variety $\Jbar = \Jac(\Xbar)$ of $\Xbar$ is an ordinary abelian variety. 
	Then, we have surjective homomorphisms 
	\[
	(\Z/p^{\Mur})^{\oplus g} \surj \piabX^{\geo}_{\ram} \surj (\Z/p^{N_J})^{\oplus g}, 
	\]
	where
	$N_J = \max\set{n | J[p^n]\subset J(k)}, M^{\ur} = \max \set{m | \mu_{p^m} \subset k^{\ur}}$,
	and $g = \dim J$.
	Here, we denoted by $k^{\ur}$ the maximal unramified extension of $k$ 
	and $\mu_{p^m}$ is the group of $p^m$-th roots of unity.
\end{thm}

\begin{rem}
\label{rem:M}
	Put $M = \max \set{m | \mu_{p^m} \subset k}$. 
	In general, 
	we have inequalities $N_J\le M \le M^{\ur}$.
	Here, the first inequality follows from the Weil pairing.  
	For the later inequality $M\le M^{\ur}$, 
	if we assume $\mu_p\subset k$, that is, $M\ge 1$ 
	and put $e_0(k)= e_k/(p-1)$, 
	where $e_k$ is the absolute ramification index of $k$, 
	then 
	$M= \Mur$ if and only if  
	$\zeta_{p^M} \not \in \Im\left(U^{pe_0(k)}_k\inj k^{\times} \surj  k^{\times}/(\kt)^p \right)$, 
	where 
	$\zeta_{p^M}$ is a primitive $p^M$-th root of unity, 
	$U_k^{pe_0(k)}$ 
	is the higher unit group (see e.g., \cite[Lemma 2.1.5]{Kaw02}).
	For example, when the base field $k$ is of the form $k = k_0(\zeta_{p^m})$ 
	for some finite unramified extension $k_0/\Qp$, 
	we have $M = M^{\ur} = m$.
	If we additionally assume $N_J = M$ as we considered in \cite{Hir21} 
	(we also give some elliptic curves satisfying this condition in \autoref{sec:EC}), 
	then the exact sequence \eqref{seq:ram} splits and 
	we have $\piabXgeoram \simeq (\Z/p^m)^{\oplus g}$. 
	One can recover the main theorem in \cite{Hir21}.
	%
\end{rem}

The above theorem enables us to construct  
an abelian geometric covering $\wt{X}\to X$ 
corresponding to $\piabXgeoram$ (\autoref{explicitetale})
along the context of the geometric abelian class field theory (\emph{e.g.}, \cite{Ser88}). 
This can be regarded as an analogue of Yoshida's work 
on the modular curve $X_0(p)$ over $\Qp$ (\cite{Yos02}). 
In \autoref{sec:EC}, we give examples in dimension $1$, 
that is when $X=E$ is an elliptic curve with good ordinary reduction, 
to indicate that each one of the two bounds given in \autoref{thm:main_intro} can be achieved 
depending on the $\Gal(\overline{k}/k)$-action on the Tate module of $X$ (\Cf \autoref{thm:2.9}). 
This in particular shows that \autoref{thm:main_intro} is as general as it can be. 
We also consider an elliptic curve $X=E$ over $k$ with good \emph{supersingular reduction}  
and we give bounds for $\piabXgeoram$ of similar flavor as in \autoref{thm:main_intro}.

\subsection*{Notation}
Throughout this note, we use the following notation: 
We fix  
 a finite extension $k$ of $\Qp$. 
For a finite extension $K/k$, we define 
\begin{itemize}
	\item $O_K$: the valuation ring of $K$ with maximal ideal $\m_K$, 
	\item $\FK = \OK/\mK$: the residue field of $K$,
	\item $G_K := \Gal(\kbar/K)$: the absolute Galois group of $K$, and 
	\item $U_K = \OK^{\times}$: the unit group of $\OK$. 
\end{itemize}
For an abelian group $G$ and $m\in \Z_{\ge 1}$,  
we write $G[m]$ and $G/m$ for the kernel and cokernel 
of the multiplication by $m$ on $G$ respectively. 
We also denote by $G\{m\} := \bigcup_{n\ge 1}G[m^n]$ the $m$-primary part of $G$. 
For a profinite group $G$, 
and a $G$-module $M$, we denote by $M^G \subset M$ and $M\surj M_G$ 
its $G$-invariant subgroup and $G$-coinvariant quotient, respectively.
In this note, by a \textbf{varitey} over $k$ we mean an integral and separated scheme of finite type over $k$,  
and a \textbf{curve} over $k$ is a variety over $k$ with dimension $1$.

\subsection*{Acknowledgements} The first author was partially supported by the NSF grant DMS-2001605. 
The second author was supported by JSPS KAKENHI Grant Number 20K03536.
The authors 
would like to give heartful thanks to Prof.~Yoshiyasu Ozeki who allowed us to include his result (\autoref{prop:ozeki}). 
We would also like to thank Prof.~Takao Yamazaki whose comments on the construction of the maximal covering 
in \autoref{sec:max} were an enormous help to us.

\section{Preliminaries}
\label{sec:Galois}

\subsection*{Finite by divisible}
Following \cite{RS00}, 
we introduce the following notation: 
\begin{dfn}[{\cite[Lemma 3.4.4]{RS00}}]
\label{def:fin-by-div}
	An abelian group $G$ is said to be \textbf{finite by divisible} 
	if $G$ has a decomposition $G\simeq F\oplus D$ 
	for a finite group $F$ and a divisible group $D$. 
	In what follows, we often denote by $G_{\fin}$ and $G_{\div}$ the subgroups of $G$ isomorphic to $F$ and $D$ respectively.
\end{dfn}

\begin{lem}[{\cite[Lemma 3.4.4]{RS00}}]
\label{lem:RS3.4.4}
	\begin{enumerate}
	\item Let $G$ be an abelian group. 
	Then, 
	$G$ is finite by divisible if and only if 
	$\ds\plim_{m\ge 1}G/m$ is finite.
	The last condition holds 
	if $G/m$ is finite for any $m\ge 1$,  
	and its order is bounded independently of $m$.
	\item 
	If $G\to G'$ is a surjective homomorphism of abelian groups, and 
	if $G$ is finite by divisible, then so is $G'$.
	\item 
	Suppose that there is a short exact sequence $0 \to G'' \to G \to G' \to 0$ 
	of abelian groups. 
	If $G$ is finite by divisible, and $G'$ is finite, 
	then $G''$ is also finite by divisible. 
	\end{enumerate}
\end{lem}
\begin{proof}
	The assertions (i), (ii) follow from \cite[Lemma 3.4.4]{RS00}. 
	
	\sn 
	(iii) 
	For any $m\ge 1$, consider the exact sequence 
	\begin{equation}
		\label{eq:Tor}
		\Tor(G',\Z/m)\to G''/m \to G/m \to G'/m \to 0 
	\end{equation}
	induced from the short exact sequence $0 \to G'' \to G \to G' \to  0$. 
	Since $G$ is finite by divisible, 
	$G/m$ is finite and its order is bounded independently of $m$. 
	From $\Tor(G',\Z/m) = G'[m] \subset G'$ and $G'$ is finite, 
	both $G'/m$ and $\Tor(G',\Z/m)$ are finite and their orders are bounded. 
	From the exact sequence \eqref{eq:Tor} 
	the same holds for $G/m$ and hence $G$ is finite by divisible from (i). 
\end{proof}

\subsection*{Mackey products, and the Galois symbol map}
We recall the definition and properties of Mackey 
functors following \cite[(3.2)]{RS00}.  
For properties of Mackey functors, see also \cite{Kah92a}, \cite{Kah92b}.

\begin{dfn}[{\Cf \cite[Section~3]{RS00}}]
\label{def:Mack}
    A \textbf{Mackey functor} $\M$ (over $k$)  (or a \textbf{$G_k$-modulation} in the sense of \cite[Definition~1.5.10]{NSW08}) 
    is a contravariant 
    functor from the category of \'etale schemes over $k$ 
    to the category of abelian groups 
    equipped with a covariant structure 
    for finite morphisms 
    such that 
    $\M(X_1 \sqcup X_2)  = \M(X_1) \oplus \M(X_2)$ 
    and if the left diagram below is Cartesian, then the right becomes commutative: 
    \[
    \xymatrix@C=15mm{
      X' \ar[d]_{f'}\ar[r]^{g'} & X \ar[d]^{f} \\
      Y' \ar[r]^{g} & Y
    }\qquad 
        \xymatrix@C=15mm{
      \M(X') \ar[r]^{{g'}_{\ast}} & \M(X)\\
      \M(Y') \ar[u]^{{f'}^{\ast}} \ar[r]^{{g}_{\ast}} & \M(Y)\ar[u]_{f^{\ast}}.
    }
    \]
\end{dfn}

For a Mackey functor $\M$, 
we denote by $\M(K)$ its value $\M(\Spec(K))$  
for a field extension $K$ of $k$.
For any finite extensions $k\subset K \subset L$, 
the induced homomorphisms 
from the canonical map $j\colon\Spec(L) \to \Spec (K)$
are denoted by 
$N_{L/K}:=j_{\ast}:\M(L) \to \M(K)$ 
which is often referred as the \textbf{norm map},\ and\  
$\Res_{L/K}:=j^{\ast}:\M(K)\to \M(L)$ is called the \textbf{restriction}.

\begin{ex}
\label{MFexs}
	\begin{enumerate}
	\item Let $G$ be a commutative algebraic group over $k$. 
	Then, $G$ induces a Mackey functor by defining $G(K) =G(\Spec K)$ for $K/k$ finite. 
	\item 
	For a Mackey functor $\M$, 
	and for $m \in \Z_{\ge 1}$, we define a Mackey functor  $\M/m$ 
	by $(\M/m) (K) := \M(K)/m$
	for any finite extension $K/k$. 
	\end{enumerate}	
\end{ex}

The category of Mackey functors forms 
an abelian category with the following tensor product:

\begin{dfn}[\Cf \cite{Kah92a}]
	\label{def:otimesM}
	For Mackey functors $\M$ and $\N$, their \textbf{Mackey product} $\M\otimesM \N$ 
	is defined as follows: For any field extension $k'/k$, 
	\begin{equation}
		\label{eq:otimesM}
		(\M\otimesM \N ) (k') := 
		\left.\left(\bigoplus_{K/k':\,\mathrm{finite}} \M(K) \otimesZ \N(K)\right)\middle/\ (\textbf{PF}),\right. 
	\end{equation}
	where \PF\ stands for the subgroup generated 
	by elements of the following form: 
	\begin{itemize}[leftmargin=!,  align=left, nosep]
	\item [\PF]
	For finite field extensions 
	$k' \subset K \subset L$, 
	\begin{align*}
		&N_{L/K}(x) \otimes y - x \otimes \Res_{L/K}(y) \quad 
		\mbox{for $x \in \M(L)$ and $y \in \N(K)$, and}\\
		&  x \otimes N_{L/K}(y) - \Res_{L/K}(x) \otimes y\quad 
		\mbox{for $x \in \M(K)$ and $y \in \N(L)$}.
	\end{align*}
	\end{itemize}
\end{dfn}
For the Mackey product  $\M\otimesM \N$,  
we write $\set{x,y}_{K/k'}$ 
for the image of $x \otimes y \in \M(K) \otimesZ \N(K)$ in the product 
$(\M\otimesM \N)(k')$. 
For any finite field extension $k'/k$, the norm map 
$ N_{k'/k} = j_{\ast}: (\M \otimesM \N )(k') \to 
(\M\otimesM \N) (k)$ is given by 
\begin{equation}	\label{eq:norm map}
	N_{k'/k}(\set{x,y}_{K/k'}) = \set{x,y}_{K/k}.
\end{equation}

Let $G$ be a semi-abelian variety over $k$. 
For any $m\in \Z_{\ge 1}$, 
the connecting homomorphism associated to the short exact sequence 
$0\to G[m]\to G \onto{m} G \to 0$ as $G_k$-modules gives, 
for each finite extension $K/k$, 
\begin{equation}\label{def:Kummer}
	\delta_{G}\colon G(K)/m \inj H^1(K,G[m]):=H^1(G_K,G[m]),
\end{equation}
which is often called the \textbf{Kummer map}. 

\begin{dfn}[{\Cf \cite[Proposition~1.5]{Som90}}]\label{def:symbol}
	For semi-abelian varieties $G_1$ and $G_2$ over $k$,
	the \textbf{Galois symbol map} 
	\begin{equation}	\label{def:sM}
		s_m\colon(G_1 \otimesM G_2)(k)/m \to H^2(k,G_1[m]\otimes G_2[m])
	\end{equation}
	is defined by the cup product and the corestriction:  
	 $s_m(\set{x,y}_{K/k}) = \mathrm{Cor}_{K/k}\left(\delta_{G_1}(x)\cup \delta_{G_2}(y)\right)$. 
	The map is well-defined by the functorial properties of Galois cohomology 
	(\Cf \cite[Proposition 1.5.3 (iv)]{NSW08}).
\end{dfn}

For semi-abelian varieties $G_1, G_2$ over $k$, the \textbf{Somekawa $K$-group} $K(k;G_1,G_2)$ 
attached to $G_1, G_2$ is a quotient of the Mackey product $(G_1\otimes  G_2)(k)$ 
(see \cite{Som90} for the precise definition)
by considering $G_1,G_2$ as Mackey functors (\Cf \autoref{MFexs}). 
By definition, 
for every $K/k$ finite there is a surjection,
$(G_1\otimesM G_2)(K)\surj K(K;G_1,G_2)$.
The elements of $K(k;G_1,G_2)$ will also be denoted as linear combinations of symbols of the form 
$\set{x_1,x_2}_{K/k}$, where $K/k$ is some finite extension and $x_i\in G_i(K)$ for $i=1,2$.
The Galois symbol map 
$s_{m}:(G_1\otimesM G_2)(k)/m \to H^2(k,G_1[m]\otimes G_2[m])$ (\autoref{def:symbol}) 
factors through $K(k;G_1,G_2)$ and the 
induced map 
\begin{equation}	\label{def:s}
	s_m\colon K(k;G_1, G_2)/m \to H^2(k,G_1[m]\otimes G_2[m])
\end{equation}
is also called the \textbf{Galois symbol map}.

\subsection*{Geometric fundamental groups, and their ``ramified parts''}
Let $V$ be a projective and smooth variety over $k$. 
We assume that there exists a $k$-rational point $x \in V(k)$. 
From this assumption, $V$ is geometrically connected. 
The abelianization of the fundamental group $\pi_1(V)$ is denoted by 
$\piab(V)$. Since we always consider the abelian fundamental groups, we omit the geometric point. 
Furthermore, 
we say that $\varphi\colon W\to V$ is an \textbf{abelian covering} 
if $\varphi$ is an \'etale covering (that is, finite and \'etale), 
and is Galois whose Galois group $\Aut(\varphi)$ is an abelian group. 
Let $k(V)$ be the function field of $V$. 
The map $\Spec(k(V)) \to V$ induces 
a surjective homomorphism 
\begin{equation}	\label{eq:GkVab}
	\Gal(\kV^{\ab}/\kV)\simeq \piab(\Spec(\kV)) \surj \piab(V), 
\end{equation}
where $\kV^{\ab}$ is the maximal abelian extension of $\kV$ (\cite[Expos\'e IX, Proposition 8.2]{SGA1}). 
We define the \textbf{maximal unramified extension} $\kV^{\urab}$ of $\kV$ by 
\begin{equation}	\label{def:abur}
	\kV^{\urab} := \bigcup_{\substack{\kV\subset F \subset \kV^{\ab}\\ 
	\mbox{\tiny unramified over $V$}}} F.
\end{equation} 
Here, a finite extension $F/\kV$ is said to be \textbf{unramified over $V$}, 
if the normalization of $V$ in $F$ is unramified over $V$, 
or equivalently, \'etale over $V$. 
The kernel of the map \eqref{eq:GkVab} is $\Gal(\kV^{\ab}/\kV^{\urab})$ 
and hence $\piabV \simeq \Gal(\kV^{\urab}/\kV)$. 
The structure map $V\to \Spec(k)$ induces 
a surjective homomorphism $\pi_1(V)\surj \pi_1(\Spec(k)) = G_k$ 
(\cite[Expos\'e IX, Th\'eor\`eme 6.1]{SGA1}).  
This map induces a short exact sequence 
\begin{equation}	\label{eq:fund_ex_seq}
	0 \to \piab(V)^{\geo} \to \piab(V) \to G_k^{\ab} \to 0, 
\end{equation}
where $\piab(V)^{\geo}$ is defined by the exactness and is called 
the \textbf{geometric fundamental group} of $V$. 
By the fixed $k$-rational point $x\in V(k)$, the above sequence splits. 
The fundamental group $\piab(V)^{\geo}$ 
classifies (abelian) \emph{geometric coverings} of $X$. 
Here, an abelian covering $\varphi \colon V'\to V$ is said to be \textbf{geometric} if 
the fiber $\varphi^{-1}(x) = V'\times_V x \to \Spec(k)$ of $\varphi$ over $x$ is \textbf{completely split}, 
in the sense that $\varphi^{-1}(x)$ is the sum of 
distinct $[k(V'):k(V)]$ $k$-rational points.
(\Cf \cite[II Preliminaries]{KL81}).
More precisely,  
the geometric fundamental group $\piabXgeo$ is written as 
\[
	\piabVgeo \simeq \Gal(\kV^{\urab}/\kV k^{\ab}) \simeq \Gal(\kV^\geo/\kV), 
\]
where 
\[
	\kV^{\geo} := \bigcup_{\substack{\kV\subset F \subset \kV^{\urab}\\ 
	\mbox{\tiny completely split over $x$}}} F.
\]
Here, a finite extension $k\subset F \subset \kV^{\urab}$ 
is said to be \textbf{completely split} over $x$ 
if the normalization of $V$ in $F$ is completely split over $x$. 

In the following, we assume that $V$ has \textbf{good reduction}, that is, 
there exists a proper smooth model over $\Ok$ of $V$. 
We denote by $\Vbar = \V \otimes_{\Ok}\Fk$ the special fiber of $\V$ which is 
a smooth variety over the finite field $\Fk$. 
In this case, it is known that $\piabVgeo$ is finite (\cite[Corollary 1.2]{Yos03}, \cite[Chapter 4]{Ras95}). 
By the valuative criterion for properness, 
the fixed rational point $x$ gives rise to an $\Ok$-rational point of $\V$ and hence to an 
$\Fk$-rational point of $\Vbar$ denoted by $\xbar$.   
In the same way as above, we have a split short exact sequence 
\[
0 \to \piab(\Vbar)^{\geo} \to \piab (\Vbar) \to G_{\Fk}\to  0. 
\]
By \cite[Expos\'e X, Th\'eor\`em 2.1]{SGA1},
there is a canonical surjection 
\begin{equation}
\label{eq:sp}
\sp\colon \piab(V)\to \piab(\V) \simeq  \piab(\Vbar)
\end{equation}
and this induces the following commutative diagram:
\begin{equation}
\label{diag:sp}	
\vcenter{
\xymatrix{
0 \ar[r] & \piab(V)^{\geo} \ar[r]\ar@{-->}[d]^{\sp} &  \piab (V) \ar[r]\ar@{->>}[d]^{\sp} & G_k^{\ab}\ar[r]\ar@{->>}[d] & 0\, \\
0 \ar[r] & \piab(\Vbar)^{\geo} \ar[r] & \piab (\Vbar) \ar[r] &  G_{\Fk}\ar[r] &  0.
}}
\end{equation}
As the horizontal sequences split, 
the specialization map 
$\sp\colon \piab(V)^{\geo}\to \piab(\Vbar)^{\geo}$ on the geometric fundamental groups 
is surjective. 

\begin{dfn}[{\cite[Definition 2.2]{Yos02}}]\label{def:piram}
	We denote by $\piab(V)_{\ram}$ (resp.\ $\piab(V)_{\ram}^{\geo}$) the kernel of the specialization map 
	$\sp\colon \piab(V) \to \piab(\Vbar)$ 
	(resp.\ $\sp\colon  \piab(V)^{\geo}\to \piab(\Vbar)^{\geo}$ on the geometric fundamental groups).
	The abelian coverings corresponding to $\piabVram$ are said to be 
	\textbf{completely ramified over $\Vbar$}. 
\end{dfn}

For the later use, we give a precise description of $\piabVram$. 
First, we recall the construction of the map $\sp\colon \piabV\to \piabVbar$: 
For an \'etale covering $\ol{\varphi}\colon\Wbar \to \Vbar$, 
there exists a unique \'etale covering $\W\to \V$ such that 
its closed fiber is $\ol{\varphi}$ 
(\cite[Expos\'e IX, Th\'eor\`eme 1.10]{SGA1}). 
By taking the generic fiber $\varphi\colon W \to V$ of $\W\to \V$,  
we obtain  
\begin{equation}
\label{eq:covers}
\vcenter{
\xymatrix{
W\ar[d]_{\varphi} \ar[r] &  \W \ar[d] & \ar[l] \Wbar\ar[d]^{\ol{\varphi}}\, \\
V\ar[r] & \V & \ar[l] \Vbar.
}}
\end{equation}
This induces the map $\sp\colon \piabV\to \piabVbar$. 
\begin{dfn}
\label{def:ur}
	For an abelian covering $\varphi\colon W\to V$ with Galois group $\Aut(\varphi) = G$, 
	we say that $\varphi\colon W\to V$ is \textbf{unramified over $\Vbar$}, 
	if there exists an abelian covering $\ol{\varphi}\colon \Wbar \to \Vbar$ with $\Aut(\ol{\varphi}) \simeq 
	G$ such that $\varphi$ and $\ol{\varphi}$ fit into the diagram \eqref{eq:covers} as above.	
\end{dfn}

We define 
\begin{equation}
\label{def:kVVbarur}
	\kV_{\Vbar}^{\ur} := \bigcup_{\substack{
	\kV\subset F \subset \kV^{\urab}\\
\mbox{\tiny unramified over $\Vbar$}}}F.
\end{equation}
Here, a finite field extension $F/\kV$ is said to be \textbf{unramified over $\Vbar$} if the normalization of $V$ 
in $F$ is unramified over $\Vbar$.  
We have $\piabVbar \simeq \Gal(\Fk(\Vbar)^{\urab}/\Fk(\Vbar)) \simeq \Gal(\kV_{\Vbar}^{\ur}/\kV)$. 
In particular, 
there is a one to one correspondence 
\begin{equation}
	\label{eq:1to1}
\set{\mbox{abelian coverings of $V$ unramified over $\Vbar$ }}
\stackrel{1:1}{\longleftrightarrow} \set{\mbox{abelian coverings of $\Vbar$}}.
\end{equation}
A diagram of fields and their Galois groups is  
\[
\xymatrix{
               &           &\kV^{\urab} \ar@{--}[rd] & \\
               &    &               & k^{\urab}\kV^{\geo}\ar@{--}[rd] &  \\
   \kV k^{\ab}\ar@/_10mm/@{-}[rrdd]_{G_k^{\ab}}\ar@{--}[rruu]\ar@/^10mm/@{-}[rruu]^{\piabV^{\geo}}\ar@{--}[rd]  &           & \kV_{\Vbar}^{\ur}\ar@{--}[uu] \ar@/^5mm/@{-}[uu]^{\piabV_{\ram}} \ar@{--}[ru]   \ar@/_5mm/@{-}[ru]_{\piabVgeoram}    &                    &  \kV^{\geo} \\
               & \kV k^{\urab}\ar@{--}[rd]\ar@{--}[ru]\ar@/^5mm/@{-}[ru]^{\piabVbar^{\geo}}  & & \\
               &           & \kV\ar@{--}[rruu]\ar@{--}[uu]\ar@/_5mm/@{-}[uu]_{\piab(\Vbar)}  
}
\]
(\Cf The diagram of fields and Galois groups in \cite[Introduction]{KL81}). 
An abelian covering $\varphi\colon W\to V$ is completely ramified over $\Vbar$ 
if and only if $\varphi$ does not have a sub covering which is unramified over $\Vbar$.

\subsection*{Class field theory for curves over a $p$-adic fields}
We keep the notation and the assumptions: $X$ is a projective smooth curve over $k$  with $X(k)\neq \emptyset$ 
and has good reduction. 
Following \cite{Blo81}, \cite{Sai85a}, we 
recall the class field theory for the curve $X$. 
The group $\SKX$ is defined by the cokernel of the tame symbol map 
\begin{equation}
	\label{def:SK1}
\SKX = \Coker\!\big(\d \colon K_2^M(k(X)) \to \bigoplus_{x} k(x)^{\times}\big), 
\end{equation}
where  
$x$ runs through the set of closed points in $X$, 
$k(x)$ is the residue field at $x$, and 
$k(X)$ is the function field of $X$. 
The norm maps $N_{k(x)/k}\colon k(x)^{\times} \to \kt$  for closed points $x$ 
induce $N\colon \SKX \to \kt$. 
Its kernel is denoted by $\VX$. 
The reciprocity map $\sigma_X\colon \SKX \to \piab(X)$ is compatible with 
the reciprocity map $\rho_k\colon \kt \to G_k^{\ab}$ of local class field theory as in the commutative diagram: 
\begin{equation}
\label{eq:rho}
	\vcenter{
	\xymatrix{
	0 \ar[r]& V(X)\ar@{-->}[d]^{\tau_X} \ar[r] &SK_1(X) \ar[r]^-{N}\ar[d]^{\sigma_X} & \kt \ar[d]^{\rho_k}\,\\
	0  \ar[r]& \piab(X)^{\geo} \ar[r]& \piab(X)\ar[r] & G_k^{\ab}\ar[r] & 0,
	}}
\end{equation}
where the bottom horizontal sequence is induced from the structure map $X\to \Spec(k)$ (\Cf \eqref{eq:fund_ex_seq}). 
The diagram above gives a map $\tau_X\colon \VX \to  \piabXgeo$ 
to describe the geometric fundamental group $\piabXgeo$.
In fact, the above short exact sequences split 
from the assumption $X(k)\neq \emptyset$. 
The main theorem of the class field theory for $X$ is the following:  

\begin{thm}[\cite{Blo81}, \cite{Sai85a}]
\label{thm:cft}
The following are true for the reciprocity maps $\sigma_X$ and $\tau_X$. 
\begin{enumerate}
	\item The reciprocity map $\sigma_X$ has dense image in $\piabX$, and 
	$\Ker(\sigma_X) = SK_1(X)_{\div}$, 
	where $SK_1(X)_{\div}$ is the maximal divisible subgroup of $SK_1(X)$. 
	\item The map $\tau_X$ is surjective, and $\Ker(\tau_X) = \VXdiv$, 
	where $\VXdiv$ is the maximal divisible subgroup of $\VX$. 
	\item $\Im(\tau_X)$ is finite.
\end{enumerate}
\end{thm}

From the above theorem, $\tau_X$ induces 
an isomorphism 
$\VX/\VXdiv \isomto \piabXgeo$ of finite groups. 
Since an extension of a finite group by a divisible group splits, 
$\VX$ is finite by divisible: $\VX = \VXfin \oplus \VXdiv$. 
Moreover, the group $\VX$ can be realized as a Somekawa $K$-group as  
\begin{equation}
	\label{thm:Som}
	 V(X) \simeq  K(k; J, \Gm)
\end{equation} 
associated with the Jacobian variety $J = \Jac(X)$ and $\Gm$
(\cite[Theorem~2.1]{Som90}, \cite[Remark~2.4.2 (c)]{RS00}). 
For $X$ has good reduction, the Jacobian variety $J$ has also good
reduction. 
The reciprocity map $\tau_X\colon \VX \to \piabXgeo$ 
coincides with the Galois symbol map 
associated with $J$ and $\Gm$ (\cite[Proposition~1.5]{Som90}) as 
in the following commutative (up to sign) diagram: 
For any $m\in \Z_{\ge 1}$, 
\begin{equation}
	\label{diag:sm}
	\vcenter{\xymatrix{
	\VX/m \ar[d]_{\simeq}^{\eqref{thm:Som}} \ar[r]^-{\tau_{X,m}} & \ar[d]^{\simeq} \piab(X)^{\geo}/m\\
	K(k;J,\Gm)/m\ar[r]^-{s_m} & H^2(k,J[m] \otimes \mu_{m})
	}}
\end{equation}
(\Cf \cite[Theorem~1.14]{Blo81}). 
Here, the right vertical map is induced from 
$H^2(k,J[m] \otimes \mu_m) \simeq J[m]_{G_k}$. 
By the class field theory for $X$ (\autoref{thm:cft}), 
the map $\tau_{X,m}$ induced from $\tau_X$ is surjective. 
 As $\Ker(\tau_X)$ is divisible, $\tau_{X,m}$ is injective. 
 We conclude that the Galois symbol $s_m$ is bijective for every $m\geq 1$.  
(Note that the injectivity of $s_m$ 
has also been established for an arbitrary field in \cite[Appendix]{Yam05}.) 

By \cite[Section 2]{KS83b}, there is a surjective homomorphism 
$SK_1(X)\to \CH(\Xbar)$, 
called  the \textbf{boundary map}, where $\CH(\Xbar)$ is the Chow group of 
the special fiber $\Xbar = \X\otimes_{\Ok}\Fk$ of the model $\X$. 
This map is compatible with the valuation map $v_k$ of $k$ as 
the following commutative diagram indicates:
\begin{equation}	\label{def:dX}
	\vcenter{\xymatrix{
	0 \ar[r] & V(X) \ar[r]\ar@{-->}[d]^{\dX} & SK_1(X) \ar[r]^-N \ar@{->>}[d] & k^{\times} \ar[r]\ar[d]^{v_k} & 0\, \\
	0 \ar[r] & A_0(\Xbar) \ar[r] & CH_0(\Xbar)\ar[r]^-{\deg} & \Z \ar[r] & 0,
	}}
\end{equation}
where $\deg$ is the degree map, and $A_0(\Xbar)$ is its kernel. 
We denote by $\dX$ the induced map $\VX\to A_0(\Xbar)$. 
Because the horizontal sequences split, the boundary map $\dX$ is surjective.
A rational point $x\in X(k)$ gives rise to an $\Fk$-rational point of $\Xbar$ 
by the valuative criterion for properness. 
The Abel-Jacobi map gives an isomorphism 
$A_0(\Xbar)\isomto \Jbar(\Fk)$, 
where $\Jbar = \Jac(\Xbar)$ is the Jacobian variety of $\Xbar$. 

\begin{lem}\label{lem:dec_KerdX}
	$\Ker(\dX)$ is finite by divisible (in the sense of \autoref{def:fin-by-div}). 
	Namely, 
	$\Ker(\dX) = \Ker(\dX)_{\fin}\oplus \Ker(\dX)_{\div}$ 
	for a finite group $\Ker(\dX)_{\fin}$ and a divisible group $\Ker(\dX)_{\div}$.  
\end{lem}
\begin{proof}
	Consider the short exact sequence 
	$0 \to \Ker(\dX) \to \VX \to A_0(\Xbar)\to 0$. 
	As noted above $\VX$ is finite by divisible 
	and $A_0(\Xbar)\simeq \Jbar(\Fk)$ is finite.
	The assertion follows from \autoref{lem:RS3.4.4} (iii). 
\end{proof}

The classical class field theory (for the curve $\Xbar$ over $\Fk$)   
says that the reciprocity map $\rho_{\Xbar}:A_0(\Xbar) \isomto \piabXbargeo = \piab(\Xbar)_{\tor}$ is bijective 
of finite groups 
and makes 
the following diagram commutative: 
\[
	\xymatrix{
	0 \ar[r] & \Ker(\dX)\ar@{-->}[d]^{\mu_X}  \ar[r] & V(X)\ar[r]^-{\dX}\ar@{->>}[d]^{\tau_X}   &  A_0(\Xbar)\ar[d]^{\rho_{\Xbar}}_{\simeq} \ar[r] & 0\,\\ 
	0 \ar[r] & \piabX^{\geo}_{\ram} \ar[r] &\piabX^{\geo}\ar[r]^-{\sp}  &  \piabXbar^{\geo} \ar[r] & 0.
	}
\]
For the commutativity of the right square in the above diagram, 
see \cite[Proposition 2]{KS83b}. 
From the diagram, we obtain the surjective homomorphism 
$\mu_X\colon \Ker(\dX) \surj \piabXgeoram$ 
with $\Ker(\mu_X) \simeq \Ker(\tau_X) = \VXdiv$. 
Since the group $A_0(\Xbar)$ is finite, we have an equality $\Ker(\dX)_{\div} = \VXdiv$.  
Moreover, $\tau_X$ induces $\VX_{\fin} = \VX/\VX_{\div}\isomto\piabX^{\geo}$. It follows 
that the reciprocity map $\mu_X$ induces an isomorphism of finite groups
\begin{equation}\label{eq:mu}	
	\Ker(\dX)_{\fin} \isomto \piabXgeo_{\ram}.
\end{equation}

\section{Abelian varieties}
\label{sec:AV}
Throughout this section, 
we will be using the following notation: 
\begin{itemize}
	\item $A$: an abelian variety over $k$ of dimension $g = \dim (A)$ 
	 with good \emph{ordinary} reduction. 
	\item $\A$: the N\'eron model over $\Ok$ of $A$ (\cite[Section 1.2]{BLR90}). 
	\item $\Abar := \A\otimes_{\Ok}\Fk$: the special fiber of $\A$ which is an ordinary abelian variety over $\Fk$. 
	\item $\Ahat$: the formal group law over $\Ok$ of $A$ (\Cf \cite[Section C.2]{HS00}). 
\end{itemize}

The formal group law $\Ahat$ defines a Mackey functor by 
the associated group  
$\Ahat(K) := \Ahat(\mK)$ for a finite extension $K/k$.

\subsection*{Boundary map}
For any $m\ge 1$, the finite flat group scheme $\A[m]$ over $\Ok$ 
fits into the following connected-\'etale exact sequence 
\begin{equation}\label{seq:conn-et_A}
	0 \to \A[m]^{\circ}  \xrightarrow{\iota} \A[m] \xrightarrow{\pi} \A[m]^{\et}\to 0
\end{equation} 
(\Cf \cite[Section 1.4]{Tat67}). 
By taking the limit $\plim_m$, we obtain 
the short exact sequence 
\begin{equation}\label{seq:conn-et_T}
	0 \to T(\A)^{\circ} \xrightarrow{\iota} T(\A) \xrightarrow{\pi} T(\A)^{\et} \to 0
\end{equation}
of the full Tate modules, 
where $T(\A)^{\bullet} := \plim_m \A[m]^{\bullet}$ 
for $\bullet \in \set{\circ, \emptyset, \et}$.
On the other hand, 
the group 
$\Ahat(\kbar) := \ilim_{k'/k} \Ahat(\m_{k'})$ 
associated with the formal group law $\Ahat$ over $\Ok$ of $A$ 
gives the short exact sequence 
\begin{equation}\label{seq:conn-et_A2}
	0 \to \Ahat[m] \xrightarrow{\iota} A[m] \xrightarrow{\pi} \Abar[m] \to 0, 
\end{equation}
where $\Ahat[m] = \Ahat(\kbar)[m]$ is the $m$-torsion subgroup of $\Ahat(\kbar)$
(\cite[Theorem C.2.6]{HS00}). 
The valuative criterion of properness yields 
$\A[m]\simeq A[m]$ as $G_k$-modules. 
By the equivalence of categories between finite \'etale group schemes over 
$\Ok$ and finite $G_k$-modules, 
we have $\A[m]^{\et} \simeq \Abar[m]$ (\Cf \cite[Section 1.4]{Tat67}). 
The group $\Ahat(\kbar)$ has no non-trivial prime to $p$-torsion (\cite[Proposition C.2.5]{HS00}). 
By comparing the short exact sequences \eqref{seq:conn-et_A} and \eqref{seq:conn-et_A2}, we obtain
\begin{equation}\label{eq:conn part}
	\A[m]^{\circ} \simeq \Ahat[m],\quad T(\A)^{\et}\simeq \plim_m\Abar[m],\quad\mbox{and}\quad T(\A)^{\circ} \simeq  \plim_{m}\Ahat[m].
\end{equation}

By taking the $G_k$-coinvariance of \eqref{seq:conn-et_T}, we have 
\begin{equation}\label{seq:con-et_Gcoinv}
	 (T(\A)^{\circ})_{G_k} \xrightarrow{\iota} T(A)_{G_k} \xrightarrow{\pi} (T(\A)^{\et})_{G_k}\to 0.   
\end{equation} 
From \cite[(3.2.1)]{Som90} (see also \cite[Remark 2.7]{Blo81}), 
the \'etale quotient of the above sequence becomes  
\begin{equation}\label{eq:TAet}
	(T(\A)^{\et})_{G_k} \simeq \plim_{m} (\Abar[m])_{G_k} \simeq \Abar(\Fk).
\end{equation}
The Galois symbol map associated to $A$ and $\Gm$ induces 
\begin{equation}\label{def:dG}
	\d_A:=\d_{A,k} \colon K(k;A,\Gm)\xrightarrow{\plim_m s_m} \plim_m H^2(k,A[m]\otimes \mu_m) \stackrel{(\diamondsuit)}{\simeq}  T(A)_{G_k} \stackrel{\pi}{\surj} \Abar(\Fk),
\end{equation}
where the middle isomorphism ($\diamondsuit$) follows from 
the local Tate duality theorem (\cite[Theorem 7.2.6]{NSW08}, \Cf \cite[(2.2)]{Blo81}, 
see also \autoref{prop:coinv} in Appendix). 
We call this map $\d_A$ the \textbf{boundary map} of $A$.
It is known that the limit of the Galois symbol map $\plim_m s_m$ in \eqref{def:dG} 
is surjective (\cite[Theorem 3.3]{Som90}), so is $\dA$.

\begin{lem}\label{lem:dec}\label{lem:l-part}
\begin{enumerate}
	\item The groups 
	$(A\otimes \Gm)(k), K(k;A,\Gm)$ and $\Ker(\dA)$ are finite by divisible in the sense of 
	\autoref{def:fin-by-div}. 
 	\item For any $m\ge 1$ prime to $p$, we have $\Ker(\dA)/m =0$.
	\end{enumerate}
\end{lem}
\begin{proof}
	(i) The proof of \cite[Theorem 4.5]{RS00} implies that 
	$(A\otimes \Gm)(k)/m$ is finite and its order is bounded independently of $m$. 
	This implies the first assertion by \autoref{lem:RS3.4.4} (i) (as in \autoref{lem:dec_KerdX}). 
	Since we have the quotient map $(A\otimes \Gm)(k) \surj K(k;A,\Gm)$, 
	the second assertion follows (\autoref{lem:RS3.4.4} (ii)). 
	
	Consider the short exact exact sequence 
	$0 \to \Ker(\dA) \to K(k;A,\Gm) \to \Abar(\Fk)\to 0$. 
	Since $\Abar(\Fk)$ is finite, \autoref{lem:RS3.4.4} (iii) implies that 
	$\Ker(\dA)$ is finite by divisible. 
	
	\sn
	(ii) 
	From (i), 
	there are decompositions $K(k;A,\Gm) = K(k;A,\Gm)_{\fin}\oplus K(k;A,\Gm)_{\div}$ 
	and $\Ker(\dA) = \Ker(\dA)_{\fin} \oplus \Ker(\dA)_{\div}$ (\Cf \autoref{def:fin-by-div}). 
	As the target of the boundary map $\dA\colon K(k;A,\Gm) \to \Abar(\Fk)$ 
	is finite, 
	we obtain a short exact sequence 
	\[
	0\rightarrow \Ker(\dA)_{\fin}\rightarrow K(k;A,\Gm)_{\fin}\xrightarrow{\dA}\Abar(\Fk)\rightarrow 0.
	\]
	Take any $m\ge 1$ coprime to $p$. 
	For 
	$K(k;A,\Gm)_{\fin}$ and $\Abar(\Fk)$
	are finite, the exact sequences 
	\begin{gather*}
	0 \to K(k;A,\Gm)_{\fin}[m] \to K(k;A,\Gm)_{\fin}\xrightarrow{m} K(k;A,\Gm)_{\fin} \to K(k;A,\Gm)_{\fin}/m \to 0,\\
	0 \to \Abar(\Fk)[m] \to \Abar(\Fk)\xrightarrow{m} \Abar(\Fk) \to \Abar(\Fk)/m \to 0
	\end{gather*}
	induce 
	\begin{equation}\label{eq:star}
	K(k;A,\Gm)_{\fin}[m]\simeq K(k;A,\Gm)_{\fin}/m \stackrel{(\star)}{\simeq}\Abar(\Fk)/m\simeq\Abar(\Fk)[m],
	\end{equation}
	where the isomorphism 
	($\star$) follows from \cite[Proposition 2.6]{Hir21} 
	(for the case where $A$ is the Jacobian variety, \cite[Proposition 2.29]{Blo81}).  
	On the prime to $p$-torsion part 
	the boundary map $\dA$ gives an isomorphism $K(k;A,\Gm)\{p'\} \isomto \Abar(\Fk)\{p'\}$. 
	This implies that $\ker(\dA)_{\fin}$ is a $p$-primary torsion group.  
\end{proof}

Let $T_p(\A)^{\bullet} = \plim_{n}(\A[p^n]^{\bullet})$ be the $p$-adic Tate module of $\A[p^n]^{\bullet}$ for $\bullet \in \set{\circ, \emptyset, \et}$ 
(\Cf \eqref{seq:conn-et_A}) 
and write $T(\A)^{\bullet} = T_p(\A)^{\bullet} \times T'(\A)^{\bullet}$ with $T'(\A)^{\bullet} = \plim_{(m,p) = 1}\A[m]^{\bullet}$. 
From the following lemma, one can describe $\Ker(\dA)_{\fin}$ by using the exact sequence 
\[
(T_p(\A)^{\circ})_{G_k}\xrightarrow{\iota} T_p(A)_{G_k} \xrightarrow{\pi} (T_p(\A)^{\et})_{G_k} \to  0, 
\]
where $T_p(A) := \plim_{n}A[p^n]\simeq T_p(\A)$
(\Cf \eqref{seq:con-et_Gcoinv}).

\begin{lem}\label{bloch}
	Suppose that, for any $m\ge 1$, 
	the Galois symbol map $s_m:K(k;A,\Gm)/m\to H^2(k,A[m]\otimes\mu_m)$ is injective. 
	We have  $\Ker(\dA)_{\fin} \simeq \Im((T_p(\A)^{\circ})_{G_k} \xrightarrow{\iota} T_p(A)_{G_k})$. 
\end{lem}
\begin{proof}
	(i) For any $m\in \Z_{\ge 1}$, 
	it follows from \cite[Theorem 3.3]{Som90} that 
	the Galois symbol map $s_m\colon K(k;A,\Gm)/m \to H^2(k,A[m]\otimes \mu_m)$ is surjective. 
	From the assumption, it is bijective.
	By taking the projective limit, 
	we obtain $\plim_m s_m=s_A\colon K(k;A,\Gm)_{\fin} \isomto T(A)_{G_k}$. 
	From the definition of the boundary map \eqref{def:dG}, 
	we have a commutative diagram 
	\[
	\xymatrix{
	K(k;A,\Gm)_{\fin} \ar@{->>}[d]^{\dA} \ar[r]^{s_A}_{\simeq}  & T(A)_{G_k}\ar@{->>}[d]^{\pi}\\
	\Abar(\Fk) \ar[r]^{\simeq}_{\eqref{eq:TAet}} &(T(\A)^{\et})_{G_k}.
	}
	\]
	This gives $\Ker(\dA)_{\fin} \simeq \Ker\left(T(A)_{G_k} \xrightarrow{\pi} \Abar(\Fk)\right)$. 
	Next, \autoref{lem:l-part} (ii) yields an isomorphism  
	 $K(k;A,\Gm)/m\simeq \Abar(\Fk)/m$ for any $m\in \Z_{\ge 1}$ which is prime to $p$. 
	Thus, we have $T'(A)_{G_k} \isomto \plim_{(m,p) = 1} (\Abar(\Fk)/m)_{G_k}$ 
	and the following commutative diagram:
	\[
		\xymatrix{
		(T_p(\A)^{\circ})_{G_k}\ar@{=}[d] \ar[r]^{\iota}& T_p(A)_{G_k} \ar@{^{(}->}[d] \ar[r] & (T_p(\A)^{\et})_{G_k} \ar[r]\ar@{^{(}->}[d]&  0\,\\
		(T_p(\A)^{\circ})_{G_k}\ar[r] & T(A)_{G_k} \ar[r] & \Abar(\Fk) \ar[r] &  0.
		}
	\]
	Here, the first vertical map is the identity, the second is the natural inclusion induced by $T_p(A)\inj T(A)$ (which splits) and the third one is the composition $(T_p(\A)^{\et})_{G_k}\simeq \Abar(\Fk)\{p\}\inj \Abar(\Fk)$ (\cite[Remark 2.7]{Blo81}), 
	where $\Abar(\Fk)\{p\}$ is the $p$-primary torsion subgroup of $\Abar(\Fk)$. 
	Then, it is clear that 
	\[
		\Im((T_p(\A)^\circ)_{G_k}\xrightarrow{\iota} T_p(A)_{G_k}) = 
		\Im((T_p(\A)^\circ)_{G_k}\to T(A)_{G_k}) = \Ker(T(A)_{G_k}\xrightarrow{\pi} \Abar(\Fk)).
	\]
\end{proof}

\subsection*{Formal groups associated with abelian varieties}
In this paragraph, we give an upper bound  for the 
Mackey product $(\Ahat \otimes \Gm)(k)$ associated to $\Ahat$ and $\mathbb{G}_m$. 

\begin{lem}\label{lem:GH4.3}
	Let $k'/k$ be a finite tamely ramified extension. 
	Then, the norm map  
	\[	
		N_{k'/k}\colon (\Ahat \otimes \Gm ) (k') \surj (\Ahat\otimes \Gm)(k)
	\]
	is surjective.
\end{lem}
\begin{proof}
	Take any symbol of the form $\set{x,a}_{K/k}$ in 
	$(\Ahat  \otimes \Gm)(k)$. 
	For $Kk'/K$ is also tamely ramified, 
	there exists $\xi\in \Ahat(K)$ such that $N_{Kk'/K}(\xi) = x$ 
	(\cite[Proposition~3.9]{CG96}). 
	The \emph{projection formula}, 
	that is, the relation \PF defining the Mackey product in \autoref{def:otimesM}, 
	yields 
	\[
	\set{x,a}_{K/k} = \set{N_{Kk'/K}(\xi), a}_{K/k} \stackrel{(\mathbf{PF})}{=} \set{ \xi,\Res_{Kk'/K}(a)}_{Kk'/k} 
	\stackrel{\eqref{eq:norm map}}= N_{k'/k}(\set{\xi,\Res_{Kk'/K}(a)}_{Kk'/k'}).
	\]
	These equations imply the assertion.
\end{proof}

In the same way as in \autoref{def:symbol}, for any $n\ge 1$, 
we define the Galois symbol map  
\begin{equation}
	\label{def:sAhat}
	s_{p^n}:= s_{p^n,k}\colon (\Ahat \otimes \Gm)(k)/p^n \to H^2(k,\Ahat[p^n] \otimes \mu_{p^n}) 
\end{equation}
by $s_{p^n}(\set{x,a}_{K/k}) = \Cor_{K/k}(\delta_{\Ahat}(x)\cup \delta_{\Gm}(a))$, 
where $\delta_{\Ahat}:\Ahat(K)/p^n \inj H^1(K,\Ahat[p^n])$ is the Kummer map. 
This map is well-defined by properties of the cup product (\cite[Proposition 1.5.3]{NSW08}). 

\begin{prop}
\label{prop:Hir20}
We assume $\Ahat[p]\subset \Ahat(k)$, $\mu_p\subset k$, and $\Abar[p] \subset \Abar(\F_k)$.
\begin{enumerate}
	\item There is an isomorphism 
	$\Ahat/p \simeq \Ubar^{\oplus g}$ 
	of  Mackey functors over $k$, 
	where $\Ubar$ is the sub Mackey of $\Gm/p$ defined by 
	\begin{equation}\label{def:Ubar}
		\Ubar(K) := \Ubar_K := \Im(U_K\to K^{\times}/p)  =  U_K/p.
	\end{equation}

	\item For any $n\ge 1$, 
	the Galois symbol map  
	\[
		s_{p^n}\colon (\Ahat \otimes \Gm)(k)/p^n \to H^2(k,\Ahat[p^n] \otimes \mu_{p^n}) 
	\]
	defined in \eqref{def:sAhat} is  bijective.
\end{enumerate}
\end{prop}
The isomorphism $\Ahat/p \simeq \Ubar^{\oplus g}$ in the assertion (i) is not canonical 
and depends on the choice of an isomorphism 
$\Ahat[p] \simeq (\mu_p)^{\oplus g}$ of (trivial) Galois modules. 
The proof of the above proposition essentially follows from \cite[Section~4]{Hir21},
but the assumptions are weakened slightly. 

\begin{proof}[Proof of \autoref{prop:Hir20}]
	(i) We fix an isomorphism 
	$\Ahat[p] \simeq (\mu_p)^{\oplus g}$ 
	of Galois modules.
	This induces the bijection ($\clubsuit$) below 
	\[
	\delta_K\colon  \Ahat(K)/p\stackrel{\delta_{\Ahat}}{\inj} H^1(K,\Ahat[p]) \stackrel{(\clubsuit)}{\simeq} H^1(K,\mu_p)^{\oplus g} \xleftarrow{\simeq} \left(K^{\times}/p\right)^{\oplus g}
	\]
	for any finite extension $K/k$. 
	Here, the last map is the Kummer map on $\Gm$ (\Cf \eqref{def:Kummer}) 
	which is bijective from ``Hilberts Satz 90''.
	First, we show 	
		$\Im(\delta_{K}) \subset (\Ubar_{K})^{\oplus g}$.
	Consider the following commutative diagram: 
	\[
	\xymatrix{
	\Ahat(K)/p \ar[r]^-{\delta_K}\ar[d] & (\Kt/p)^{\oplus g} \ar[r]^-{v}\ar[d]^-{\iota} & (\Z/p)^{\oplus g}\ar[d]^-{\Id}   \\ 
	\Ahat(K^{\ur})/p\ar[r]^{\delta_{K^{\ur}}} & \left((K^{\ur})^{\times}/p\right)^{\oplus g} \ar[r]^-{v}  & (\Z/p)^{\oplus g},
	}
	\]
	where $K^{\ur}$ is the completion of the maximal unramified extension of $K$, 
	and $v$ is the valuation map. 
	Since we have $\Ahat\otimes_{\Ok}O_{k^{\ur}} \simeq (\widehat{\mathbb{G}}_m)^{\oplus g}$ (\cite[Lemma~4.26, Lemma~4.27]{Maz72}), 
	$\Ahat(K^{\ur})/p \simeq (\Ubar_{K^{\ur}})^{\oplus g}$ and 	
	the composition $v\circ \delta_{K^{\ur}} = 0$ in the above diagram. 
	Thus, the composition 
	$v\circ \delta_K = 0$ 
	in the top sequence and hence $\Im(\delta_{K}) \subset 
	(\Ubar_{K})^{\oplus g}$. 
	From the structure theorem of the multiplicative group $K^{\times}$, 
	we have  $U_K/p \simeq (\Z/p)^{\oplus ([K:\Qp]+1)}$ 
	and hence $\# (\Ubar_K)^{\oplus g}=  \{\#(U_K/p)\}^g = p^{g([K:\Qp]+1)}$. 
	It is enough to show $\#\Ahat(K)/p \ge p^{g([K:\Qp]+1)}$. 
	
	Mattuck's theorem (\cite{Mat55}) and $\#A(K)[p]= p^{2g}$ imply 
	$\#A(K)/p = p^{g([K:\Qp]+2)}$. 
	Recall that 
	$\Abar$ has ordinary reduction 
	so that 	$\Abar [p] \simeq (\Z/p)^{\oplus g}$. 
	The exact sequence 
	\[
	\Ahat(K)/p \to A(K)/p \to \Abar(\F_K)/p \to 0
	\]
	and the equality 
	$\#\Abar(\F_K)/p = \#\Abar(\F_K)[p]$ 
	imply the inequality 
	$\#\Ahat(K)/p\ge p^{g([K:\Qp]+1)}$.
	The map $\delta_{K}\colon \Ahat(K)/p \isomto (\Ubar_K)^{\oplus g}$ is bijective.
	
	\sn 
	(ii)
	For each $n\in \Z_{\ge 1}$, 
	to simplify the notation, 
	we put $\M_n := (\Ahat \otimes \Gm )(k)/p^n$, 
	$\H_n := H^2(k,\Ahat[p^n]\otimes \mu_{p^n})$ 
	and  
	$s_n:= s_{p^n}\colon  \M_n\to \H_n$.  
	We will show by induction that $s_n$ is bijective. 
	First, we show that $s_1\colon \M_1 \to \H_1$ is bijective. 
	As in the proof of (i) above, we fix an isomorphism $\Ahat[p]\simeq (\mu_p)^{\oplus g}$ of Galois modules
	and hence we obtain  
	\begin{equation}\label{eq:H1}
		\H_1 = H^2(k,\Ahat[p]\otimes \mu_p) \simeq H^2(k,\mu_p^{\otimes 2})^{\oplus g}. 
	\end{equation}
	By (i), there is an isomorphism  $\Ahat/p \simeq \Ubar^{\oplus g}$. 
	For the Mackey product commutes with the direct sum, 
	\begin{equation}\label{eq:M1}
		\M_1 \simeq (\Ahat/p\otimes \Gm/p)(k) \simeq (\Ubar \otimes \Gm/p)(k)^{\oplus g}.
	\end{equation}
	The natural inclusion 
	$\Ubar \hookrightarrow \Gm/p$,  
	induces the following commutative diagram:
	\[	
	\xymatrix@C=5mm{
		\M_1\ar[d]^{\simeq}_{\eqref{eq:M1}}\ar[rr]^-{s_{1}} & & \H_1 \ar[d]_{\eqref{eq:H1}}^-{\simeq} \\
	(\Ubar\otimes \Gm/p)(k)^{\oplus g}\ar[r] & (\Gm/p\otimesM \Gm/p)(k)^{\oplus g}\ar[r]^-{(s_p)^{\oplus g}} &H^2(k,\mu_p^{\otimes 2})^{\oplus g}.
	 }
	\]
	Here, the map $s_p$ in the bottom is 
	the Galois symbol map associated to two $\Gm$.  
	In fact, the composition 
	 $(\Ubar\otimes \Gm/p)(k) \to (\Gm/p\otimesM \Gm/p)(k) \xrightarrow{s_p} H^2(k,\mu_p^{\otimes 2})$ is 
	bijective (\cite[Lemma 4.2.1]{RS00}, see also \cite[Lemma~4.5]{Hir21})
	and so is $s_1\colon \M_1\to \H_1$.
		 
	Next, we consider the following commutative diagram with exact rows 
	except possibly at $\M_{n-1}$: 
	\[
		\xymatrix@C=3mm@R=0mm{
		\Ahat[p] \otimesZ \kt \ar[rr]^-{\psi}\ar[dd]^{\phi} & 
			& \M_{n-1} \ar[rr]\ar[dd]^{s_{n-1}}_{\simeq}   & 
			& \M_{n} \ar[rr]\ar[dd]^-{s_{n}} & & \M_1\ar[dd]_{\simeq}^-{s_{1}}\ar[rr] &  & 0 \\
			& (\diamondsuit) & \\
		H^1(k, \Ahat[p] \otimes \mu_{p}) \ar[rr] & & \H_{n-1} \ar[rr] &
			& \H_{n} \ar[rr] & & \H_1 \ar[rr] & & 0
	}
	\]
	(\Cf \cite[Lemma 4.2.2]{RS00}), 
	where the bottom sequence is induced from 
	\[
	0 \to \Ahat[p^{n-1}]\otimes \mu_{p^n} \to \Ahat[p^{n}]\otimes \mu_{p^n} \to \Ahat[p]\otimes \mu_p\to 0. 
	\]
	Here, the far left vertical map $\phi$ is given by 
	\[
		\Ahat[p] \otimesZ \kt \xrightarrow{\Id\otimes \delta} 
			H^0(k,\Ahat[p])\otimesZ H^1(k,\mu_p) \onto{\cup} H^1(k,\Ahat[p]\otimes \mu_p)
	\]
	and $\psi$ is induced from $\Ahat[p]\inj \Ahat(k) \surj \Ahat(k)/p^{n-1}$: 
	$\psi(w\otimes a) := \set{w,a}_{k/k}$ for $w\otimes a \in \Ahat[p] \otimes \kt$. 
	The commutativity of the square ($\diamondsuit$) follows from 
	a property of the cup product (\Cf \cite[Proposition 1.4.3 (i)]{NSW08}).
	By the fixed isomorphism $\Ahat[p]\simeq (\mu_p)^{\oplus g}$ 
	of trivial Galois modules, 
	the map $\phi$ becomes 
	\[
		\Ahat[p] \otimesZ \kt  \surj (\mu_p \otimesZ \kt/p)^{\oplus g} \simeq 
			H^1(k,\mu_p^{\otimes 2})^{\oplus g} \simeq H^1(k,\Ahat[p]\otimes \mu_p).
	\]
	In particular, $\phi$ is surjective. 
	From the inductive hypothesis, $s_{n-1}$ is bijective and hence $s_n$ is surjective. 
	From the diagram chase and the induction hypothesis, 
	$s_{n}$ is injective. 	
\end{proof}

\begin{thm}
\label{thm:GmA}
	For any $n\ge 1$, 
	there is a surjective homomorphism 
	\[
		(\Z/p^{\Mur})^{\oplus g} \surj  (\Ahat \otimes \Gm)(k)/p^n,
	\] 
	where $\Mur = \max\set{m \ge 0 | \mu_{p^m}\subset k^{\ur}}$.
\end{thm}
\begin{proof}	
	For any finite unramified extension $k'/k$, 
	the norm map 
	$(\Ahat \otimes \Gm)(k')  \to (\Ahat \otimes \Gm)(k)$ 
	is surjective (\autoref{lem:GH4.3}). 
	We may assume $M^{\ur} = M := \max\set{m\ge 0 | \mu_{p^m}\subset k}$.  
	We have a short exact sequence $0 \to \Ahat[p] \to A[p]\to \Abar[p] \to 0$ 
	by \cite[Theorem C.2.6]{HS00}. 
	Mazur's theorem $\Ahat\otimes_{\Ok}O_{k^{\ur}} \simeq (\widehat{\mathbb{G}}_m)^{\oplus g}$ (\cite[Lemma~4.26, Lemma~4.27]{Maz72}) 
	indicates that,  
	by replacing $k$ with a finite unramified extension, 
	the above sequence becomes 
	$0 \to (\mu_p)^{\oplus g} \rightarrow A[p]\rightarrow (\Z/p)^{\oplus g}\rightarrow 0$ 
	as $G_k$-modules. 
	In particular, we have  $\Abar[p]\subset \Abar(\F_k)$. 
	All the assumptions in \autoref{prop:Hir20} are satisfied. 
	In the following, we put $K = k(\mu_p)$.

	\sn 
	\textbf{The case $M=0$:}
	First, we consider the case $M = 0$ and show $(\Ahat \otimes \Gm)(k)/p  = 0$. 
	This implies that $(\Ahat \otimes \Gm)(k)$ is $p$-divisible 
	so that $(\Ahat \otimes \Gm)(k)/p^n = 0$ for any $n \ge 1$. 
	The assumption $M=0$ implies $\mu_p\not\subset k$ and $k\subsetneqq K$. 
	Using $\Ahat[p] \simeq (\mu_p)^{\oplus g}$, 
	the Galois symbol map defined in \eqref{def:sAhat} is of the form: 
	\[
	s_p\colon  (\Ahat \otimes \Gm)(k)/p  \to H^2(k,\Ahat[p]\otimes \mu_p) \simeq H^2(k,\mu_p^{\otimes 2})^{\oplus g}.
	\]
	Since we have $H^2(k,\mu_p^{\otimes 2}) \simeq K_2^M(k)/p = 0$ 
	(\Cf \cite[Chapter IX, Proposition 4.2]{FV02}), 
	it is left to show that the Galois symbol map 
	$s_p$ is injective. 
	The extension degree of $K = k(\mu_p)/k$ is prime to $p$.  
	The composition 
	\[
		(\Ahat \otimes \Gm)(k)/p \xrightarrow{\Res_{K/k}} 
		(\Ahat \otimes \Gm)(K)/p \xrightarrow{N_{K/k}} (\Ahat \otimes \Gm)(k)/p
	\]
	is the multiplication by $[K:k]$ and is bijective. 
	The restriction $\Res_{K/k}\colon (\Ahat \otimes \Gm)(k)/p \to (\Ahat \otimes \Gm)(K)/p $ is injective. 
	Consider the following commutative diagram: 
	\[
	\xymatrix{
	(\Ahat \otimes \Gm)(k)/p  \ar@{^{(}->}[d]_-{\Res_{K/k}}\ar[r]^{s_p} & H^2(k,\mu_p^{\otimes 2})^{\oplus g}\ar[d]^{\Res_{K/k}} \\ 
	(\Ahat \otimes \Gm)(K)/p  \ar[r]_-{\simeq}^{s_{p,K}} & H^2(K,\mu_p^{\otimes 2})^{\oplus g}.
	}
	\]
	Here, the Galois symbol map $s_{p,K}$ is bijective from \autoref{prop:Hir20} (ii).
	From the diagram above, the Galois symbol map 
	$s_p$ is injective. 
	We obtain $(\Ahat \otimes \Gm)(k)/p^n =0$.

	\sn
	\textbf{The case $M>0$:}
	Next, consider the case $M>0$. 
	In this case, $K=k$. 
	Fix $\zeta\in \mu_{p^{M}}$ a primitive $p^{M}$-th root of unity. 
	In the following, we show the following claim: 

	\begin{oneclaim}
		$(\Ahat \otimes \Gm)(k)/p$ is generated by 
		symbols of the form $\set{ w,\zeta}_{k/k}$ for some $w\in \Ahat(k)$.
	\end{oneclaim} 
	\begin{proof}
		Recall that the Hilbert symbol 
		$(-,-)_p\colon k^{\times}\otimes k^{\times} \to \mu_p\simeq \Z/p$ satisfies  
		\begin{equation}
		\label{eq:H}
		(y,x)_p = 0
		\Leftrightarrow y \in N_{k(\sqrt[p]{x}\,)/k}\left(k(\sqrt[p]{x}\,)^{\times}\right),\quad \mbox{for $x,y\in k^{\times}$}
			\end{equation}
		(\Cf \cite[Proposition~4.3]{Tat76}). 
		From the very definition of $M$ and $M = \Mur$, the extension $L := k(\mu_{p^{M+1}})/k$ 
		is non-trivial, and totally ramified. 
		We have $U_k/N_{L/k}U_L \simeq k^{\times}/N_{L/k}L^{\times}$ (\Cf the proof of \cite[Section~V.3,\ Corollary~7]{Ser68}) and 
		local class field theory says 
		$k^{\times}/N_{L/k}L^{\times} \simeq \Gal(L/k)\neq 0$  (\Cf \cite[Section~XIII.3]{Ser68}). 
		Thus, there exists $y\in U_k \ssm N_{L/k}U_L$ such that 
		$(y,\zeta)_p \neq 0$ from \eqref{eq:H}. 
		As $(y,\zeta)_p\neq 0$, the chosen element $y$  induces a non-trivial element in $\Ubar_k = U_k/p$. We use the same notation $y$ 
		for this induced element in $\Ubar_k$. 
		For each $1\le i \le g$, put 
		$y^{(i)} := (1,\ldots , 1, \stackrel{i}{\stackrel{\vee}{y}}, 1,\ldots , 1) \in (\Ubar_k)^{\oplus g}$  
		and we denote by $w^{(i)}\in \Ahat(k)/p$ 
		the element corresponding to $y^{(i)}$ through the isomorphism 
		$\Ahat(k)/p \simeq (\Ubar_k)^{\oplus g}$ (\autoref{prop:Hir20} (i)).
		The Galois symbol map is compatible with the Hilbert symbol map 
		(\cite[Section~XIV.2, Proposition~5]{Ser68})
		as the following commutative diagram indicates: 
		\begin{equation}\label{diag:H-AV}
			\vcenter{\xymatrix{
			\Ahat(k)/p \otimesZ \kt/p\ar[r]^-{\iota}\ar[d]^{\simeq} & (\Ahat \otimes \Gm)(k)/p \ar[r]^-{s_{p}}_-{\simeq} \ar[r]& H^2(k,\Ahat[p]\otimes \mu_p) \ar[d]^-{\simeq} \\
			\left(\Ubar_k \otimesZ \kt/p\right)^{\oplus g} 
			\ar[rr]^-{(-,-)_p}&  &(\Z/p)^{\oplus g}.
			}}
		\end{equation}
		Here, $s_p$ is the Galois symbol map and is bijective (\autoref{prop:Hir20} (ii)), 
		and the map $\iota$ is given by 
		$\iota( w\otimes x) = \set{w,x}_{k/k}$. 
		The image of $ w^{(i)}\otimes\zeta \in \Ahat(k)/p\otimesZ \kt/p$ in $(\Z/p)^{\oplus g}$ 
		via the lower left corner in \eqref{diag:H-AV} 
		is
		$\xi^{(i)} := (0,\ldots , 0, \stackrel{i}{\stackrel{\vee}{(y,\zeta)_p}}, 0, \ldots , 0)
		\in (\Z/p)^{\oplus g}$. 
		These elements $\xi^{(i)}$ $(1\le i \le g)$ generate $(\Z/p)^{\oplus g}$ 
		and hence the symbols $\set{w^{(i)}, \zeta}_{k/k} = \iota (w^{(i)} \otimes \zeta )$ for $1\le i \le g$ 
		generate $(\Ahat \otimes \Gm)(k)/p$. 
	\end{proof}

 	For any $n\ge 1$, consider the exact sequence  
	\[
	(\Ahat \otimes \Gm)(k)/p \onto{p^n} (\Ahat \otimes \Gm)(k)/p^{n+1} \to  (\Ahat \otimes \Gm)(k)/p^n \to  0, 
	\]
	where $p^n$ is the map induced from the multiplication by $p^n$. 
	From the claim above, the map $p^n$ becomes $0$ for all $n\ge M$,
	so that $(\Ahat \otimes \Gm)(k)/p^{n+1} \simeq (\Ahat \otimes \Gm)(k)/p^{n}$. 
	It is left to show 
	$(\Z/p^M)^{\oplus g } \surj  (\Ahat \otimes \Gm)(k)/p^{M}$. 
	From \autoref{lem:GH4.3}, 
	by replacing $k$ with a sufficiently large unramified extension of it,  
	we may assume $\Ahat[p^{M}] \simeq (\mu_{p^M})^{\oplus g}$ as $G_k$-modules. 
	As the Galois symbol map 
	 $(\Ahat \otimes \Gm)(k)/p^M\to H^2(k,\Ahat[p^{M}]\otimes \mu_{p^M})$ 
	 is bijective (\autoref{prop:Hir20}, (ii)) and  $\mu_{p^M} \subset k$, we have  
	 \[
	 (\Ahat \otimes \Gm)(k)/p^M \simeq  H^2(k,\Ahat[p^{M}]\otimes \mu_{p^M}) \simeq H^2(k,\mu_{p^M}^{\otimes 2})^{\oplus g} \simeq (\Z/p^M)^{\oplus g}.
	 \]
\end{proof}

\subsection*{Upper and lower bounds of the kernel of the boundary maps}
The Mackey functor defined by 
the formal group law $\Ahat$ associated to $A$ 
gives the short exact sequence as Mackey functors 
\begin{equation}	\label{seq:E}
	0\to \Ahat \xrightarrow{\iota} A \xrightarrow{\pi} A/\Ahat \to 0,
\end{equation}
where $A/\Ahat$ is defined by the exactness. 
The Mackey functor $A/\Ahat$ is given by 
$(A/\Ahat)(K) \simeq \Abar(\FK)$ 
for each finite extension $K/k$ with residue field $\FK$  
(for the precise description, see \cite[(3.3)]{RS00}). 
By applying $- \otimesM \Gm$ (which is right exact) 
to the sequence \eqref{seq:E}, 
we have the following commutative diagram with exact rows 
\begin{equation}\label{diag:partial}
	\vcenter{\xymatrix{
	& (\Ahat \otimesM \Gm) (k) \ar[d]^{\varphi} \ar[r]^{\iota\otimes \Id} &   (A \otimesM \Gm)(k)\ar@{->>}[d] \ar[r] & ((A/\Ahat)  \otimesM \Gm) (k) \ar[r]\ar@{-->>}[d]^{\psi}  & 0 \\
	0 \ar[r] & \Ker(\dA) \ar[r] & K(k;A,\Gm) \ar[r]^-{\dA} \ar[r] & \Abar(\Fk) \ar[r] & 0, 
	}}
\end{equation}
where the middle vertical map is the quotient map, and 
$\dA$ is the boundary map defined in \eqref{def:dG}.  
Here, the commutativity of the left square in \eqref{diag:partial} 
follows from the lemma below and 
this induces the right vertical map $\psi$ which is surjective.

\begin{lem}
	The boundary map $\dA$ annihilates the image of 
	$(\Ahat\otimes \Gm)(k)$ in $K(k;A,\Gm)$.
\end{lem}
\begin{proof}
	For $m = \#\Abar(\Fk)$, 
	there is a commutative diagram:
	\[
	\xymatrix{
	(\Ahat\otimes \Gm)(k) \ar[r]\ar[d]^{\bmod m} & K(k;A,\Gm)\ar[d]^{\bmod m} \ar[r]^-{\dA} & \Abar(\Fk) \ar[d]_{\simeq}^{\bmod m} \\
	(\Ahat\otimes \Gm)(k)/m \ar[r] & K(k;A,\Gm)/m\ar[r]^-{\d_{A,m}} & \Abar(\Fk)/m. 
	}
	\]
	It is enough to show that 
	the bottom sequence is a complex. 
	The Galois symbol maps induce the following commutative diagram with exact rows:
	\[
	\xymatrix{
	(\Ahat \otimes \Gm)(k)/m \ar[r]^{\iota \otimes \Id}\ar[d]^{s_m} & (A\otimes \Gm)(k)/m\ar[d]^{s_m} \ar[r]\ar[rd]^{\d_{A,m}} & ((A/\Ahat)\otimes \Gm)(k) \ar[r]\ar@{-->}[d] & 0\,\\ 
	H^2(k,\Ahat[m]\otimes \mu_m)\ar[r]\ar[d]^{\simeq}& H^2(k,A[m]\otimes \mu_m) \ar[r]\ar[d]^{\simeq} & H^2(k,\Abar[m]\otimes \mu_m)\ar[d]^{\simeq}\ar[r] & 0\, \\
	\Ahat[m]_{G_k} \ar[r]^{\iota} & A[m]_{G_k} \ar[r]^{\pi} & \Abar[m]_{G_k} \ar[r] & 0,
	}
	\]
	where the second exact sequence is induced from the exact sequence for $A[m]$ 
	noted in \eqref{seq:conn-et_A2}. 
	The definition of the boundary map $\dA$ (\Cf \eqref{def:dG}) says that 
	the composition 
	\[
	(A\otimes \Gm)(k)/m \surj K(k;A,\Gm)/m \xrightarrow{s_m} H^2(k,A[m]\otimes \mu_m) \xrightarrow{\pi}  H^2(k,\A[m]^{\et}\otimes \mu_m)
	\]
	is the boundary map $\d_{A,m}$.
	Since the bottom sequence in the above diagram is exact, 
	$\d_{A,m}$ annihilates the image of $(\Ahat\otimes \Gm)(k)/m$ and the assertion follows from this.
\end{proof}

\begin{thm}\label{thm:main}
	There are surjective homomorphisms 
	\[
	(\Z/p^{M^{\ur}})^{\oplus g} \surj \Ker(\dA)_{\fin} \surj (\Z/p^{N_A})^{\oplus g}, 
	\]
	where 
	$N_A = \max\set{n \ge 0| A[p^n] \subset A(k)}$ and  $g = \dim(A)$.
\end{thm}
\begin{proof}
	\textbf{(Lower bound)}  
	To give the lower bound, we may assume $N := N_A>0$. 
	The diagram \eqref{diag:partial} induces  
	\[
		\xymatrix{
		(\Ahat \otimesM \Gm) (k)/p^{N} \ar[d] \ar[r] &   (A \otimes \Gm )(k)/p^{N} \ar[d]\ar[r] & ((A/\Ahat)  \otimesM \Gm) (k)/p^{N} \ar[r]\ar[d]  & 0\ \\
		\Ker(\dA)/p^{N} \ar[r] & K(k;A,\Gm)/p^{N} \ar[r]^{\d_{A,p^N}} \ar[r] & \Abar(\Fk)/p^{N} \ar[r] & 0. 
		}
	\]
	In fact, the middle and right vertical maps are bijective 
	(\cite[Lemma 4.1, Corollary 4.3 (i)]{Hir21}), 
	the upper sequence is left exact, 	and 
	$(\Ahat \otimes \Gm)(k)/p^{N} \simeq (\Z/p^{N})^{\oplus g}$  
	(\cite[Lemma 4.5, (ii)]{Hir21}). 
	Therefore, 
	\begin{equation}\label{eq:lower}
		\Ker(\dA) \surj \Ker(\dA)/p^{N}\surj \Ker(\d_{A,p^N}) \simeq (\Ahat \otimesM \Gm)(k)/p^{N}\simeq (\Z/p^{N})^{\oplus g}.
	\end{equation}
	
	\sn
	(\textbf{Upper bound}) 
	Consider the decomposition 
	$\Abar(\Fk) = \Abar(\Fk)\{p\} \oplus \Abar(\Fk) \{m\}$ 
	for some $m$ coprime to $p$. 
	The composition $\d_A^{\{p\}}: K(k;A,\Gm)\xrightarrow{\dA} \Abar(\Fk) \surj \Abar(\Fk)\{p\}$ 
	gives the following diagram:
	\begin{equation}\label{eq:d_Ap} 
		\xymatrix{
		0 \ar[r] & \ar@{^{(}->}[d]^{j} \Ker(\dA) \ar[r] & K(k;A,\Gm) \ar@{=}[d]\ar[r]^{\dA} & \Abar(\Fk) \ar@{->>}[d]\ar[r] & 0\, \\
		0 \ar[r] & \Ker(\dA^{{\{p\}}}) \ar[r] & K(k;A,\Gm) \ar[r]^{\dA^{\{p\}}} & \Abar(\Fk)\{p\} \ar[r] & 0.
		}
	\end{equation}
	By applying the snake lemma, the above diagram induces an isomorphism 
	$\Abar(\F)\{m\}\isomto  \Coker(j)$. 
	Since $\Tor_\Z(\ck(j),\Z/p^n)=0$, we conclude that 
	\begin{equation}\label{eq:KerdA} 
		\Ker(\dA)/p^n \isomto \Ker(\dA^{\{p\}})/p^n.
	\end{equation}
	From the diagram \eqref{diag:partial}, we have 
	\begin{equation}\label{diag:partialp} 
		\vcenter{\xymatrix{
		& (\Ahat \otimesM \Gm) (k) \ar[d]^{\varphi^{{\{p\}}}} \ar[r] &   (A \otimesM \Gm)(k)\ar@{->>}[d] \ar[r] & ((A/\Ahat)  \otimesM \Gm) (k) \ar[r]\ar[d]^{\psi^{{\{p\}}}}  & 0 \\
		0 \ar[r] & \Ker(\dA^{\{p\}}) \ar[r] & K(k;A,\Gm) \ar[r]^-{\dA^{\{p\}}} \ar[r] & \Abar(\Fk)\{p\} \ar[r] & 0, 
		}}
	\end{equation}
	where the right vertical map $\psi^{{\{p\}}}$ is 
	the composition $((A/\Ahat)  \otimesM \Gm) (k)\xrightarrow{\psi} 
	\Abar(\Fk) \surj \Abar(\Fk){\{p\}}$.

	\begin{oneclaim}
		$\Ker(\psi^{\{p\}})$ is $p$-divisible. 
	\end{oneclaim}
	\begin{proof}
		Put $\M = ((A/\Ahat\,) \otimes \Gm)(k)$. 
		Since $\psi^{\{p\}}$ induces an isomorphism 
		\[
		\M/p^n \simeq \Abar(\Fk){\{p\}}/p^n = \Abar(\Fk)/p^n
		\] 
		for all $n\ge 1$ (\cite[Lemma 4.1]{Hir21}), we have
		$\plim_n \M/p^n \simeq \plim_n \Abar(\Fk)/p^n \simeq \Abar(\Fk){\{p\}}$.  
		It follows that 
		\begin{equation}
		\label{eq:Kerpsi}
		\Ker(\psi^{\{p\}}) = \Ker\left(\M\to \plim_n\M/p^n\right) = \bigcap_{n\ge 1}p^n\M. 
		\end{equation}  
		As $\Abar(\Fk)\{p\}$ is a finite $p$-group, 
		there exists $s\ge 0$ such that $p^s$ annihilates $\Abar(\Fk)\{p\}$. 
		To show the claim, take any $x\in \Ker(\psi^{\{p\}})$ and any $n\ge 1$. 
		From \eqref{eq:Kerpsi}, 
		there exists $y \in \M$ such that $x = p^{n+s}y = p^n(p^sy)$. 
		Here, $p^sy \in \Ker(\psi^{\{p\}})$. 
		Thus, $\Ker(\psi^{\{p\}})$ is $p$-divisible. 
	\end{proof}
 
	The snake lemma applied to the diagram \eqref{diag:partialp} yields a surjection 
	$\Ker(\psi^{{\{p\}}})\surj \Coker(\varphi^{\{p\}})$. 
	From the above claim, $\Coker(\varphi^{\{p\}})$ is also $p$-divisible. 
	The map $\varphi^{\{p\}}$ induces a surjective homomorphism
	\[
	\varphi_n: (\Ahat \otimesM \Gm) (k)/p^n \stackrel{\varphi^{\{p\}}}{\surj} \Ker(\dA^{\{p\}})/p^n \stackrel{\eqref{eq:KerdA}}{\simeq} \Ker(\dA)/p^n.
	\]
	From \autoref{thm:GmA}, 
	we obtain 
	\begin{equation}
	\label{eq:p-part}	
	(\Z/p^{M^{\ur}})^{\oplus g} \surj (\Ahat\otimes \Gm)(k)/p^n \overset{\varphi_n}{\surj} \Ker(\dA)/p^n 
	\end{equation}
	for any $n \ge 1$. 
	For the finite part $\Ker(\dA)_{\fin}$ is a $p$-group 
	(\autoref{lem:l-part} (ii))  
	this implies the existence 
	of surjective homomorphism
	$(\Z/p^{M^{\ur}})^{\oplus g} \surj \Ker(\dA)_{\fin}$
	as required.
\end{proof}

\begin{rem}
\label{rem:upper_E}
	In the case where $A = E$ is an elliptic curve, define 
	\[
	\Nhat := \max \set{n | \Ehat[p^n]\subset \Ehat(k)}.
	\] 
	In general, we have $N \le \Nhat$. 
	By \cite[Lemma 4.26 and Lemma 4.27]{Maz72}, 
	the base change $\Ehat[p^n]_{\kur}$ to $\kur$ 
	gives $\Ehat_{\kur}[p^{n}] \simeq \mu_{p^n}$ and hence 
	$\Nhat \le  M^{\ur}$. 
	Using this, 
	we will give a refined upper bound 
	$\Z/p^{\Nhat} \surj \Ker(\dE)$ 
	 in \autoref{prop:upper_E}.
\end{rem}

\begin{rem}
	As noted in the introduction, 
	we apply \autoref{thm:main} 
	to the Jacobian variety $J = \Jac(X)$ for 
	a curve $X$ over $k$ which has good reduction  
	to obtain the structure of $\piabXgeoram$ (\autoref{cor:main}). 
	However, the structure of $\Ker(\d_J) \subset K(k;J,\Gm) \simeq V(X)$ 
	can be obtained 
	without assuming $X$ has good reduction. 
	Precisely, 
	let $X$ be a projective smooth curve over $k$  with $X(k)\neq \emptyset$ 
	and assume that the Jacobian variety $J = \Jac(X)$ has good ordinary reduction. 
	From \autoref{thm:main}
	there are surjective homomorphisms 
	\[
	(\Z/p^{M^{\ur}})^{\oplus g} \surj \Ker(\dJ)_{\fin} \surj (\Z/p^{N_J})^{\oplus g}.
	\]
	Note that when $X$ has good reduction 
	(this is the very case studied in \cite{Blo81}), 
	its Jacobian $J$ has good reduction. But, the converse does not hold in general.  
	By the semi-stable reduction theorem, at least $X$ has semi-stable reduction, 
	that is, there exists a model $\X$ over $\Ok$ of $X$ whose 
	closed fiber $\Xbar = X\otimes_{\Ok}\Fk$ is semistable, \Ie $\Xbar$ is reduced and has at most
	ordinary double points as singularities
	(\cite[Theorem 2.4]{DM69}).
\end{rem}

The following proposition due to Yoshiyasu Ozeki
insists that 
if we enlarge the base field $k$ 
then the difference $N_A\le \Mur$ becomes arbitrarily large.

\begin{prop}\label{prop:ozeki}
	Let $A$ be an abelian variety over $k$ with potentially good reduction. 
	For an extension $K/k$, we define
	\begin{align*}
	N_A(K) &:=  \max \set{n | A[p^n] \subset A(K)} = N_{A_K},\ \mbox{and}\\
	M(K) &:= \max \set{m | \mu_{p^n} \subset K^{\times}}. 
	\end{align*}
	Then, 
	for any $x > 0$,
	there exists a finite extension $K/k$, such that 
	$M(K)-N_A(K) > x$. 
\end{prop}
\begin{proof}
	For each  $m\ge 1$, put 
	$k_m := k(\mu_{p^m})$ and $k_{\infty} := \bigcup_{m\ge 1}k_m$.
	By definition, for any $m\ge 1$, we always have 
	\begin{equation}
		\label{M}
		m \le M(k_m).
	\end{equation}
	By Imai's theorem \cite{Ima80}, 
	$\# A(k_{\infty})_{\tor}<\infty$.
	In particular, $N_A(k_{\infty})< \infty$.
	For sufficiently large $m>0$, we have 
	$A(k_{\infty})[p^{\infty}] = A(k_m)[p^{\infty}]$.
	Take such $m$ satisfying 
	\begin{equation}
	\label{s} m >N_A(k_{\infty}).
	\end{equation} 
	On the other hand, for any $t\ge 1$, 
	\[
	A[p^t]\subset A(k_\infty)
	\Leftrightarrow A[p^t] \subset A(k_{\infty})[p^{\infty}] = A(k_m)[p^{\infty}]
	\Leftrightarrow A[p^t] \subset A(k_m).
	\]
	From these equivalences, 
	\begin{equation}
	\label{N2}
	A[p^{N_A(k_{\infty})+1}] \not\subset A(k_m), \quad \mbox{and}\quad 
	A[p^{N_A(k_{\infty})}]\subset A(k_m). 
	\end{equation}
	Thus we obtain 
	\[
	N_A(k_m) \stackrel{\eqref{N2}}{=} N_A(k_{\infty}) \stackrel{\eqref{s}}{<} m \stackrel{\eqref{M}}{\le} M(k_m).
	\]
	As $N_A(k_{\infty})$ does not depend on $m$ and we can take arbitrary large $m$, 
	the assertion follows by putting $K = k_m$.
\end{proof}

\section{Curves over a $p$-adic field}
\label{sec:max}
In this section, 
we give a proof of \autoref{thm:main_intro} 
and also construct the maximal covering of a curve $X$ over $k$ 
which produces all the subgroup $\piabXgeoram$ of $\piabXgeo$. 
Throughout this section, we use the following notation:
\begin{itemize}
	\item $X$: a projective smooth curve over $k$ with $X(k)\neq \emptyset$ 
	and we additionally assume that $X$ has \emph{good reduction}. 
	\item $\Xbar:= \X\otimes_{\Ok}\Fk$: 
	the special fiber of the regular model $\X$ over $\Ok$ of $X$. 
	\item 
	$J = \Jac(X)$: the Jacobian variety of $X$ which has good reduction from the assumption on $X$, 
	\item $\J$: the N\'eron model over $\Ok$ of $J$. 
	\item $\Jbar := \Jac(\Xbar)$: the Jacobian variety of $\Xbar$ 
	which is also  the closed fiber of $\J$. 
\end{itemize}
Finally, we suppose that  $\Jbar$ is an \emph{ordinary} abelian variety.
From this assumption, the Jacobian variety $J$ has good ordinary reduction. 
We fix a rational point $x\in X(k)$. 
By the valuative criterion for properness, 
$x$ gives rise to an $\Fk$-rational point of $\Xbar$ which is denoted by $\xbar \in \Xbar(\Fk)$.

\subsection*{Proof of the main theorem}
The boundary map $\d_J$ for $J$ 
defined in \eqref{def:dG} 
is compatible with $\dX$ defined in \eqref{def:dX} 
as in the following commutative diagram: 
\begin{equation}
	\label{def:dXdJ}
	\xymatrix{
	\VX \ar@{->>}[r]^-{\dX}\ar[d]_{\eqref{thm:Som}}^{\simeq} & A_0(\Xbar) \ar[d]^{\simeq} \\
	K(k;J,\Gm) \ar@{->>}[r]^-{\dJ} & \Jbar(\F),
	}
\end{equation}
where 
the right vertical map is the Abel-Jacobi map $A_0(\Xbar) \isomto \Jbar(\Fk)$ which is bijective 
(\cite[Lemma 2.2]{Som90}, see also \cite[Lemma 2.12]{Blo81}). 
Recall that both of $\dX$ and $\dJ$ are surjective, 
we obtain an isomorphism $\Ker(\dX)\isomto \Ker(\dJ)$. 
This isomorphism and \autoref{thm:main}
together with the class field theory $\mu_X\colon \Ker(\dX)_{\fin}\isomto \piabXgeoram$ (\Cf \eqref{eq:mu}) induce 
the following main result referred in \autoref{thm:main_intro}: 	
\begin{cor}
\label{cor:main}
We have surjective homomorphisms: 
\[
(\Z/p^{\Mur})^{\oplus g} \surj \piabXgeoram \surj (\Z/p^{N_J})^{\oplus g}.
\]
\end{cor}

When the absolute ramification index $e_k = e_{k/\Qp}$ of $k$ satisfies $e_k < p-1$, 
we have 
$\mu_p \not\subset k^{\ur}$, and this implies $\Mur=0$. 
From \autoref{cor:main}  
we recover the following assertion in \cite[Proposition 7]{KS83b} 
(\Cf \cite[Theorem 3.2, Theorem 4.1]{Yos02}. For more general results, see also \cite[Proposition 4.25]{Ras95}).

\begin{cor}\label{georam}
	Assume $e_k<p-1$. Then, 
	we have
	$\piabXgeoram = 0$.
\end{cor}

\subsection*{Construction of the maximal covering} 
In the following, we construct 
a geometric covering $\varphi\colon \wt{X} \to X$ 
such that  
the composition 
\[
  \piab(X)^{\geo}_{\ram} \inj \piabXgeo \simeq \Gal(\kX^{\geo}/\kX) \surj \Aut(\varphi)
\] 
is bijective. 
The construction of such covering is known classically as 
the pullback of an appropriate isogeny $\wt{J}\to J$ along 
the Albanese map $f^x\colon X\to J = \Jac(X)$ associated with the given rational point $x\in X(k)$ 
(\Cf \cite{Ser88}). 
Since we could not find appropriate references, we give precise explanations below: 
Consider also the Albanese map $f^{\xbar}\colon \Xbar \to \Jbar$ 
(\cite[Section 6]{MilneJac}). 
We have the middle vertical arrow in the commutative diagram below 
\begin{equation}\label{diag:NMP}
	\vcenter{\xymatrix{
	X\ar[r]\ar[d]_{f^x} & \X\ar@{-->}[d] & \ar[l] \Xbar \ar[d]^{f^{\xbar}} \\
	J\ar[r] & \J & \ar[l] \Jbar
	}}
\end{equation}
by the N\'eron mapping property of $\J$.

\begin{lem}
	The diagram \eqref{diag:NMP} above induces
	$\piabXgeo \simeq \pi_1(J)^{\geo}$ and $\piabX^{\geo}_{\ram} \simeq \pi_1(J)^{\geo}_{\ram}$.
	Note that all finite \'etale coverings of $J$ are abelian.
\end{lem}
\begin{proof}
	Because of $H^2(k,\Q) = H^3(k,\Z) = 0$,  
	and the long sequence arising from $0 \to \Z\to \Q \to \Q/\Z \to 0$, 
	we have $H^2(k,\QZ) = 0$. 
	The five-term exact sequence induced by the Hochschild-Serre spectral sequence gives 
	short exact sequences
	\[
	\xymatrix{
	0 \ar[r] & H^1(k,\QZ)\ar[r]\ar[d]^{\simeq} &  H^1_{\et}(X,\QZ)\ar[r]\ar[d] & H^1_{\et}(X\otimes_k\kbar,\QZ)^{G_k}\ar[r] \ar[d]^{\simeq} & 0 \\
	0 \ar[r] & H^1(k,\QZ)\ar[r] &  H^1_{\et}(J,\QZ)\ar[r] & H^1_{\et}(J\otimes_k\kbar,\QZ)^{G_k}\ar[r]& 0. \\
	}
	\] 
	The sequences are exact on the right because the group $H^2(k,\QZ)$ vanishes. 
	Here, the right vertical map is bijective, 
	because $f^x$ induces an isomorphism $\piab(X\otimes_k\kbar) \simeq \pi_1(J\otimes_k \kbar) = \piab(J\otimes_k \kbar)$ (\cite[Proposition 9.1]{MilneJac}). 
	We obtain $\piabX\simeq H^1_{\et}(X,\QZ)^{\vee}\simeq H^1_{\et}(J,\QZ)^{\vee}\simeq \pi_1(J)$. 
	In the same way, we also obtain $\piab(\Xbar) \simeq \pi_1(\Jbar)$. 
	Thus, we obtain $\piabXgeo \simeq  \pi_1(J)^{\geo}$ and $\piabXgeoram \simeq \pi_1(J)^{\geo}_{\ram}$. 
\end{proof}

It follows by \autoref{cor:main} that there is an isomorphism
\[ 
	\piabXgeoram \simeq\bigoplus_{i=1}^g\Z/p^{r_i},
\] 
for some integers $N_J\leq r_i\leq M^{\ur}$, $i=1,\ldots,g$. 
In particular, this implies that $\piabXgeoram$ has a subgroup isomorphic to $(\Z/p^{N_J})^{\oplus g}$.  
We wish to find an explicit finite abelian covering $X'\rightarrow X$ 
whose Galois group coincides with  the aforementioned subgroup of $\piab(X)_{\ram}^{\geo}$. 
This is of course only interesting when $N_J\geq 1$. 
Put $N := N_J$ and suppose $N \geq 1$. 
Consider the splitting 
\[
	J[p^{N}]\simeq\Jhat[p^{N}]\oplus\Jbar[p^{N}]\simeq(\mu_{p^{N}})^{\oplus g}\oplus(\Z/p^{N})^{\oplus g}
\] 
induced by the connected-\'{e}tale short exact sequence for $J$ (\Cf \eqref{seq:conn-et_A2}). 
Put $H_N := \Jbar[p^N]$ and consider it as a subgroup of $J[p^N]$. 
This induces an isogeny $\psi:J\to J/H_N=:J_N$ with kernel $H_N$ (\cite[Example 4.40]{EMvdG}). 
Let $\check{\psi}:J_N\to J$ be its dual (\cite[Proposition 5.12]{EMvdG}).  

\begin{prop}\label{lem:geometric_J}
	The isogeny $\check{\psi} :J_N\to J$ is a geometric covering which is completely unramified over $\Jbar$. 
	Furthermore, we have $\Aut(\check{\psi}) \simeq (\Z/p^N)^{\oplus g}$.
\end{prop}
\begin{proof}
	\textbf{(Abelian covering)} 
	It is known that 
	any isogeny on abelian varieties is finite flat (\cite[Proposition 5.2]{EMvdG}) 
	and we are working over a characteristic 0 field, hence 
	the isogeny $\check{\psi}:J_N\to J$ is finite \'{e}tale 
	(\cite[Proposition 5.6]{EMvdG}). 
	The map  
	$\Ker(\check{\psi}) \to \Aut(\check{\psi})$ which sends
	$\xi \in \Ker(\check{\psi})$ to the automorphism given by the translation by $\xi$ 
	is bijective, 
	because any non-constant homomorphism is the composition of an isogeny and a translation by some $\xi$ (\cite[Proposition 1.14]{EMvdG}). 
	Since $\Aut(\check{\psi})$ acts transitively on the fibers $\Ker(\check{\psi})$,  
	the covering $\check{\psi}$ is Galois with Galois group $\Aut(\check{\psi})\simeq \Ker(\check{\psi}) \simeq (\Z/p^N)^{\oplus g}$. 
	
	\sn
	\textbf{(Geometric covering)} 
	Next, we show that $\check{\psi}$ is a geometric covering of $J$. 
	As we recalled in \autoref{sec:Galois}, 
	using the zero $0_J \in J(k)$, 
	it suffices to show that the fiber $(J_N)_0$ over $0_J$ 
	\[
	\xymatrix{
	J_N\ar[d]_{\check{\psi}} & \ar[d]\ar[l] (J_N)_0 = J_N \times_{J} 0_J \\
	J & \ar[l]_{0_J} \Spec(k)
	}
	\]
	is completely split over $\Spec (k)$. 
	In fact, 
	we have $(J_N)_0  \simeq \ker(\check{\psi})$ 
	as schemes and the later $\Ker(\check{\psi})$ 
	is precisely the subgroup $\psi(\Jhat[{p^N]})$, which is $k$-rational by assumption. 
	Therefore, 
	$(J_N)_0$ is the sum of $k$-rational points, and hence  
	$\check\psi\colon J_N\to J$ is a geometric covering of $J$.

	\sn
	\textbf{(Completely ramified)} 
	Finally, we show that the geometric covering $\check{\psi}\colon J_N\to J$ is completely ramified over $\Jbar$. 
	Suppose that $\check{\psi}$ contains a sub covering $\phi\colon A\to J$ 
	unramified over $\Jbar$. 
	Since the isogeny $\check{\psi}$ maps $0$ in $J_N$ to $0$ in $J$, 
	there exists a rational point $e\in A(k)$ such that $\phi (e) = 0$. 
	From the Lang-Serre theorem (\cite[Theorem 10.36]{EMvdG}), 
	$A$ is an abelian variety. 
	Let $\A$ be the N\'eron model of $A$ and $\Abar$ its closed fiber. 
	By the functorial property of the N\'eron models (\cite[Section 7.3, Proposition 6]{BLR90})  
	there exists an isogeny $\Phi\colon \A\to \J$ which makes the following diagram commutative: 
	\begin{equation}
		\label{diag:phi}
		\vcenter{
		\xymatrix{
		A\ar[d]_{\phi} \ar[r] & \A \ar[d]_{\Phi}  & \Abar\ar[d]^{\ol{\phi}}\ar[l]\\ 
		J \ar[r] & \J & \ar[l] \Jbar.
	}}
	\end{equation}
	
	\begin{oneclaim}
		The isogenies $\ol{\phi}$ and $\Phi$ are \'etale. 
		In particular, in the correspondence \eqref{eq:1to1}, 
		$\phi$ comes from the above diagram \eqref{diag:phi} with the isogeny $\Phi\colon \A\to \J$ 
		of the N\'eron models.
	\end{oneclaim}
	\begin{proof}
		The kernel $\Ker(\Phi)$ of the induced isogeny $\Phi$ 
		is a finite group scheme (\cite[Proposition 5.2]{EMvdG}). 
		Consider the connected-\'etale sequence 
		$0 \to \Ker(\Phi)^{\circ} \to \Ker(\Phi) \to \Ker(\Phi)^{\et} \to 0$ 
		(\cite[Proposition 4.45]{EMvdG}). 
		We can factor $\Phi$ as a composition of two isogenies $\A \to \A/\Ker(\Phi)^{\circ} \xrightarrow{\Phi^{\et}}\J$. 
		In the same way, $\ol{\phi}$ can be written 
		$\Abar \to \Abar/\Ker(\ol{\phi})^{\circ} \xrightarrow{\ol{\phi}^{\et}} \Jbar$. 
		Putting $\A^{\et} := \A/\Ker(\Phi)^{\circ}$ and $\Abar^{\et} := \Abar/\Ker(\ol{\phi})^{\circ}$, 
		they make the following diagram commutative:
		\[
		\xymatrix@C=3mm@R=3mm{
		A\ar[rr]\ar[dd]_{\phi} \ar@{-->}[rd] & & \A\ar[rd] \ar[dd] & & \ar[ll] \Abar\ar[rd]\ar[dd] \\
		& A^{\et}\ar@{-->}'[r][rr]\ar@{-->}[ld]^-{\phi^{\et}} &  & \A^{\et}\ar[ld]^-{\Phi^{\et}} & & \ar@{->}'[l][ll]\Abar^{\et}\ar[ld]^-{\ol{\phi}^{\et}} \\
		J\ar[rr] & & \J  & & \ar[ll] \Jbar & ,\\ 
		}
		\]
		where $\phi^{\et}\colon A^{\et}\to J$ is given by taking the generic fiber of $\Phi^{\et}$. 
		Here, $\Phi^{\et}$ and $\phi^{\et}$ are isogenies whose kernels are \'etale group schemes
		so that $\Phi^{\et}$ and $\phi^{\et}$ are \'etale (\cite[Proposition 5.6]{EMvdG}). 
		From this, $\phi^{\et}$ is an abelian covering of $J$ which is unramified over $\Jbar$. 
		
		Since $\phi$ is unramified over $\Jbar$ (and $A\to A^{\et}$ is not unramified over $\Abar^{\et}$), 
		we have $A\simeq A^{\et}$. 
		This implies that $\A \simeq \A^{\et}$ and $\Abar\simeq \Abar^{\et}$ and the assertions follow.
	\end{proof} 
	
	Let, $\J_N$ be the N\'eron model of $J_N$. 
	Extending the diagram \eqref{diag:phi}, 
	we have the following commutative diagram: 
	\[
	\xymatrix{
	J_N \ar[r]\ar[d]\ar@/_10mm/[dd]_{\check{\psi}}   & \ar[d]\J_N \ar@{-->}@/_10mm/[dd] & \ar[l] \Jbar_N \ar@/^10mm/[dd]^{\ol{\check{\psi}}}\ar[d] \\
	A \ar[r]\ar[d]^{\phi} & \A \ar[d] & \ar[l] \Abar\ar[d]_{\ol{\phi}} \\
	J\ar[r] & \J & \ar[l]\Jbar.
	}
	\]
	From the functorial property of N\'eron models, 	 
	the above diagram is commutative. 
	Here, $\ol{\phi}$ is \'etale. 
	From the construction of $J_N$, 
	$\ol{\check{\psi}}\colon \Jbar_N\isomto  \Jbar$ is an isomorphism 
	and so is $\ol{\phi}$. 
	This implies that $\phi\colon A\to J$ is an isomorphism. 
	Therefore, $\check{\psi}$ does not contain sub abelian coverings of $J$ which are unramified over $\Jbar$. 
\end{proof}

It follows (see \emph{e.g.},\ \cite[Section 9]{MilneJac}) that the pull-back 
\[
	\xymatrix{
	X_N\ar[r]\ar[d]_{\varphi} & J_N\ar[d]^{\check{\psi}}\\
	X \ar[r]^{f^x} & J
	}
\] 
of $\check{\psi}$ along $f^x\colon X\to J$ defines an \'etale covering of $X$. 
From the construction of $X_N$ and the universal property of the Albanese map $f^x$, we have 
$\Aut(\check{\psi}) \simeq \Aut(\varphi)$.

\begin{thm}\label{explicitetale}	
	Suppose we have $\Ker(\dX) \simeq (\Z/p^{N_J})^{\oplus g}$ with $N := N_J \ge 1$. 
	The \'etale covering $\varphi:X_N\to X$ 
	is a geometric covering 
	which is completely ramified over $\Xbar$. 
	Furthermore, the composition
	\[
		\piabXgeoram\inj \piabXgeo\surj \Aut(\varphi)
	\]
	is bijective.
\end{thm}
\begin{proof}
	From \autoref{lem:geometric_J}, the right vertical map in the following commutative diagram is surjective
	\[
	\xymatrix{
	\Aut(\varphi) \ar[r]^{f^x}_{\simeq} & \Aut(\check{\psi}) \\
	\piabXgeo \ar[r]^{f^x}_{\simeq} \ar[u] & \pi_1(J)^{\geo}.\ar@{->>}[u].
	}
	\]
	Thus, the left vertical map is surjective, 
	and hence $\varphi\colon X_N\to X$ is a geometric (abelian) covering of $X$. 
	
	Recall that we have $(\Z/p^N)^{\oplus g} \simeq \Ker(\dX) \simeq \piabXgeoram$ 
	and $\Aut(\varphi)\simeq \Aut(\check{\psi}) \simeq (\Z/p^N)^{\oplus g}$ (\autoref{lem:geometric_J}). 
	Consider the following commutative diagram:
	\[
	\xymatrix{
	\piabXgeoram \ar@{^{(}->}[r]\ar[d]^{\simeq} &  \piabXgeo \ar@{->>}[r]\ar[d]^{\simeq} & \Aut(\varphi)\ar[d]^{\simeq} \\
	\pi_1(J)^{\geo}_{\ram} \ar@{^{(}->}[r] &  \pi_1(J)^{\geo} \ar@{->>}[r] & \Aut(\check{\psi}). 
	}
	\]
	From \autoref{lem:geometric_J}, the composition of the bottom maps is bijective, 
	so is the top map. 
	This implies that $\varphi\colon X_N\to X$ is completely ramified over $\Xbar$ and is maximal.  
\end{proof}


\begin{rem}
	The assumption in \autoref{explicitetale} holds 
	if we have $N_J = M^{\ur}$ (see \autoref{rem:M}). 
	In \autoref{thm:2.9} below, we also consider elliptic curves 
	which satisfy this assumption. 
\end{rem}

\subsection*{Products of curves}
The above results can be extended to products of curves. 
For a product $X = X_1\times \cdots\times X_d$ of 
smooth and projective curves $X_i$ over $k$ 
with good reduction and $X_i(k)\neq \emptyset$ for all $i$, 
we have a short exact sequence
$0 \to V(X) \to SK_1(X) \onto{N} k^{\times} \to 0$  
defined similarly as in \eqref{eq:rho}. 
There is a commutative diagram  
\begin{equation}
	\xymatrix@R=4mm{
  	V(X) \ar[r]^-{\simeq}\ar[d]^{\tau_X} & \ds \bigoplus_{i=1}^dV(X_i)\ar[d]^{\oplus \tau_{X_i}}  \oplus \wt{V}(X)\\ 
  	\piabXgeo   \ar[r]^-{\simeq }&  \ds \bigoplus_{i=1}^d \piab(X_i)^{\geo}, 
  	}
\end{equation}
where $\wt{V}(X)$ is a divisible group 
(\cite[Proposition 1.7 and Corollary 2.5, see also the proof of Theorem 1.1]{Yam09}).
From the decomposition of $\VX$, 
one define the boundary map 
\[
	\d_X \colon  V(X)\xrightarrow{\mathrm{projection}} 
	\bigoplus_{i=1}^d V(X_i)\xrightarrow{\oplus \d_{X_i}} \bigoplus_{i=1}^d A_0(\Xbar_i)
	\simeq \bigoplus_{i=1}^d \Jbar_i(\Fk),
\]
where $\Jbar_i$ is the Jacobian variety of the special fiber $\Xbar_i$ for each $i$.
Here, the target of the boundary map $\dX$ can be considered 
as the Albanese variety $\Alb(\Xbar)(\Fk) = \bigoplus_i \Jbar_i(\Fk)$, 
where $\Xbar = \Xbar_1 \times \cdots \times \Xbar_d$. 
This induces the commutative diagram with horizontal exact sequence: 
\[
	\xymatrix@R=4mm{
	0\ar[r] & \Ker(\dX) \ar[r]\ar@{-->>}[d]^{\mu_X} & V(X)\ar[r]\ar@{->>}[d]^{\tau_X} &  \ds\bigoplus_{i=1}^d \Jbar_i(\Fk)\ar[r] \ar[d]_{\simeq}^{\oplus \rho_{\Xbar_i}} & 0\\
	0\ar[r] & \piabXgeoram \ar[r] & \piabXgeo \ar[r]&  \ds \bigoplus_{i=1}^d \piab(\Xbar_i)^{\geo} \ar[r] & 0. 
	}
\]
From the top horizontal sequence, 
we have a decomposition $\Ker(\dX)\simeq \Ker(\dX)_{\fin}\oplus \Ker(\dX)_{\div}$ 
(\autoref{lem:RS3.4.4} (iii)), 
with $\Ker(\dX)_{\fin} \simeq \bigoplus_i \Ker(\d_{X_i})_{\fin}$ and 
$\Ker(\dX)_{\div} = \wt{V}(X)$. 
Since $\mu_X$ induces an isomorphism 
\begin{equation}
	\Ker(\dX)_{\fin} \simeq \bigoplus_{i=1}^d\Ker(\d_{X_i})_{\fin} \isomto 
	\bigoplus_{i=1}^d\piab(X_i)^{\geo}_{\ram} \simeq \piab(X)^{\geo}_{\ram}, 
\end{equation}
\autoref{thm:main} gives  the following corollary. 

\begin{cor}
	Let $X = X_1\times \cdots\times X_d$ be a 
	product of smooth and projective  curves over $k$ 
	 with good reduction, and $X_i(k)\neq \emptyset$ for all $1\le i\le d$. 
	Assume that the Jacobian variety $\Jbar_i := \Jac(\Xbar_i)$ has ordinary reduction for each $1\le i \le d$. 
	Then, 
	there are surjective homomorphisms 
	\[
	\bigoplus_{i=1}^d (\Z/p^{\Mur})^{\oplus g_i} 
	\surj \piabXgeoram \surj \bigoplus_{i=1}^d (\Z/p^{N_{J_i}} )^{\oplus g_i},
	\]
	where $g_i = \dim(J_i)$.
\end{cor}

\section{Elliptic curve}
\label{sec:EC}
In this section, 
we consider an elliptic curve $X = E$ over $k$ 
which has good reduction. 
Recalling from \autoref{lem:dec_KerdX}, 
we have a decomposition $\Ker(\dE) \simeq \Ker(\dE)_{\fin}\oplus \Ker(\dE)_{\div}$.
We will obtain a sharp computation of the group $\Ker(\dE)_{\fin}$ under 
some mild assumptions on $E$. 
\emph{From now on} we will simply write $N$ for the integer $N_E$.

\subsection*{Good ordinary reduction}
First, we assume that $E$ has good ordinary reduction.  \autoref{thm:main} gives surjections
\[\Z/p^{M^{\ur}}\twoheadrightarrow\ker(\dE)_{\fin}\twoheadrightarrow\Z/p^N.\] 
Recall that we have the invariants 
\begin{equation}\label{def:Nhat}
	\Nhat  = \max \set{m\ge 0 | \Ehat[p^m]\subset \Ehat(k)}, \ 
	\mbox{and}\ M = \max\set{m\ge 0 | \mu_{p^m}\subset k}. 
\end{equation}
In general, we have $N \le \Nhat \le \Mur$ as noted in \autoref{rem:upper_E}.

\begin{lem}\label{lem:Tcoinv}
	Let $G\subset G_k$ be a closed subgroup, 
	and $T$ a  free  $\Zp$-module of rank $1$ with non-trivial $G$-action 
	$\chi:G \to \Aut(T)$.
	Then, $T_G \simeq \plim_{n}[(T/p^n)_G] \simeq T/p^{M_G}$,
	where $M_G = \max\set{m| \mbox{$G$ acts on $T/p^m$ trivially}}$.
\end{lem}
\begin{proof}
	Put $T_n := T/p^n$ and $m := M_G$. 
	Take a generator $(z_n)$ of $\plim_n T_n = T$ with $z_n \in T_n$. 
	We will show that, for any $n\ge m$, the natural map $T_n\surj T_m$ induces $(T_n)_G \isomto T_m$. 
	The mod $p^n$-representation $\chi_n:G\xrightarrow{\chi} \Aut(T) \surj \Aut(T_n)$ factors through a finite cyclic subgroup $G_n \subset G$. 
	Fix a generator $\sigma_n$ of $G_n$. 
	Thus, $(T_n)_G = T_n/I_G(T_n)$, where $I_G(T_n) := \Braket{(\chi_n(\sigma_n) -1)x | x\in T_n}$.
	Then $\chi_n(\sigma_n)(z_n) = a_nz_n$ for some $a_n \in (\Z/p^n)^{\times}$. 
	Since $G$ acts on $T_m$ trivially, 
	$a_nz_n \bmod p^m = z_n \bmod p^m$ in $T_m$ and hence $a_n \bmod p^m = 1$. 
	Write $a_n -1 = p^m l_n$. This equality means precisely that the subgroup 
	$I_G(T_n) \subset p^{m}T_n$. 
	To prove the reverse inclusion it is enough to show that $(l_n,p)=1$. 
	Assume for contradiction that $p \mid l_n$. 
	This yields 
	\[
	\chi_n(\sigma_n)z_n \bmod p^{m+1} = a_nz_n \bmod p^{m+1} = z_n \bmod p^{m+1}\ \mbox{in $T_{m+1}$}.
	\]
	But this means that $G$ acts trivially on $T_m$, 
	which contradicts the definition of the integer $m=M_G$. 
	
	To finish the proof we consider the following commutative diagram with exact rows,
	\[
		\xymatrix{
		0\ar[r] & I_G(T_n) \ar[r]\ar[d] & T_n \ar[r]\ar@{=}[d] & (T_n)_G\ar[r]\ar[d] & 0\\
		0\ar[r] & p^mT_n\ar[r]& T_n\ar[r] & T_m\ar[r] & 0.
		}
	\] 
	The first two vertical maps are equalities, giving the desired isomorphism $(T_n)_G\simeq T_m$. 
	In the appendix we prove an isomorphism $T_{G}\simeq\varprojlim\limits_n \left[(T_n)_{G}\right]$ (\Cf \autoref{prop:coinv}). 
\end{proof}

\begin{prop}\label{prop:upper_E}
	There are surjective homomorphisms
	\[
	\Z/p^{\Nhat} \surj \Ker(\dE)_\fin \surj \Z/p^N.
	\]
	In particular, we have $\Ker(\dE)_{\fin} \simeq \Z/p^N$ if $N = \Nhat$. 
	The inequality $N\le \Nhat$ can be strict. 
\end{prop}
\begin{proof}  
	From \autoref{bloch}
	 we have an isomorphism
	\[
	\Ker(\dE)_\fin \simeq \Im((T_p(\E)^{\circ})_{G_k}\xrightarrow{\iota} T_p(E)_{G_k}).
	\]  
	Note that the injectivity of the Galois symbol map follows from \eqref{diag:sm}. 
	By the definition of $N$, $G_k$ acts on $E[p^N]$ trivially and so does on $\Ehat[p^N]$. 
	We obtain 
	\[
	 \Im((T_p(\E)^{\circ})_{G_k}\xrightarrow{\iota} T_p(E)_{G_k}) \simeq \Im(\Ehat[p^N] \inj E[p^N]) \simeq \Z/p^N.
	\]
	From \autoref{lem:Tcoinv}, 
	$(T_p(\E)^{\circ})_{G_k} \simeq \Ehat[p^{\Nhat}] \simeq \Z/p^{\Nhat}$ 
	and this implies the assertion.
	It is clear that if $\Ehat[p^{\Nhat}]\not\subset\Ebar(\Fk)$,  the inequality $N\leq\Nhat$ becomes strict. 
\end{proof}

Let $\E$ be the N\'{e}ron model of $E$. 
For every $n\geq 1$, consider the connected-\'{e}tale exact sequence of finite flat group schemes over $\Spec(\Ok)$ (\Cf \eqref{seq:conn-et_T}), 
\begin{equation}\label{ses2}
	0\rightarrow\E[p^n]^\circ\rightarrow\E[p^n]\rightarrow\E[p^n]^{\et}\rightarrow 0.
\end{equation}
When $E$ has complex multiplication, \eqref{ses2} splits 
(\cite[A.2.4]{Ser89}). 
Equivalently, the $G_k$-action on $E[p^n]$ is
diagonal for all $n \ge 1$. 
We will refer to this as the \textbf{semisimple case}. 
In general \eqref{ses2} 
does not split and the $G_k$-action on $E[p^n]$ is upper triangular. 
Over $k^{\ur}$ the sequence \eqref{ses2} becomes 
\begin{equation}\label{ses6}
	0\rightarrow \mu_{p^n}\rightarrow \E[p^n]\rightarrow \Z/p^n\rightarrow 0.
\end{equation}  Passing to the limit we obtain a short exact sequence of continuous $G_{k^{\ur}}$-modules 
\begin{equation}
\label{ses1}
	0 \to \Z_p(1)\to T_p(E) \to \Zp\to 0.
\end{equation}
When $E$ has complex multiplication, \eqref{ses1} splits; that is, $T_p(E)$ is semisimple as $G_{k^\ur}$-module.
Suppose we are in the non-semisimple case. Assume additionally that $\mu_{p^n}\subset k$ and that $E[p^n] \subset E(\Fk)$ for some $n$. 
Then the sequence \eqref{ses6} is given over $k$ 
In particular, the group scheme $\E[p^n]$ defines an element of $\mathcal{E}xt^1_{\mathcal{O}_k}(\Z/p^n,\mu_{p^n})\simeq H^1_{fppf}(\Ok,\mu_{p^n})$. This group is isomorphic to $\Ok^{\times}/p^n$ and therefore the extension $\E[p^n]$ (or equivalently the Galois module $E[p^n]$) corresponds to a unit $u\in\Ok^{\times}/p^n$. 
That is, the sequence \eqref{ses6} becomes split after extending to the finite extension $k(\sqrt[p^n]{u})$. The unit $u$ is known as the \textbf{Serre-Tate parameter} of $E$ and it is trivial when $E$ has complex multiplication. For more information we refer to \cite[Chapter 8, Section 9]{KM85}.

\begin{thm}\label{thm:2.9}
	Let $\rho_n\colon G_k\to \Aut(E[p^n])$ be the mod $p^n$ representation arising from 
	$E[p^n]$ for any $n\ge 1$. 
	\begin{enumerate}
	\item If $\rho_{\Nhat}$ is semisimple, then $\Ker(\dE)_{\fin}\simeq \Z/p^{\Nhat}$. 
	\item 
	If $\rho_{\Nhat}$ is non semisimple, 
	we further assume
	that $M=M^{\ur}$, $\Ebar[p^M]\subset\overline{E}(\F_k)$ and 
	the restriction $\rho_{N+1}|_{I_k}$ of the mod $p^{N+1}$ representation $\rho_{N+1}$ to the inertia subgroup $I_k \subset G_k$ is also non semisimple.
	Then,  we have $\Nhat = M$, and an isomorphism $\ker(\dE)_{\fin}\simeq\Z/p^N$. That is, the lower bound is achieved  and the inequality $N\leq M=M^{\ur}$ can be strict.
	\end{enumerate}
\end{thm}

\begin{proof}
	If $N=\Nhat$ there is nothing to show, 
	so we assume $N<\Nhat$. 	
	
	\sn 
	(i)
	As in the proof of \autoref{prop:upper_E}, 
	$\Ker(\dE) \simeq \Im( (T_p(\E)^{\circ})_{G_k} \xrightarrow{\iota} T_p(E)_{G_k})$ 
	and $T_p(\E)^{\circ} \simeq \Z/p^{\Nhat}$.
	From the assumption, 
	the sequences \eqref{ses2}  are split for all $n\geq 1$ and hence 
	$(T_p(\E)^{\circ})_{G_k} \xrightarrow{\iota} T_p(E)_{G_k}$ is injective. 
	This implies that $\Z/p^{\Nhat} \simeq T_p(\E)^{\circ} \simeq \Ker(\dE)$.
	
	\sn 
	(ii) 
	Consider the short exact sequence 
	\begin{equation}\label{seq:EpM}	
		0 \to \Ehat[p^M]\to E[p^M]\to \Ebar[p^M]\to 0
	\end{equation}
	as $G_k$-modules from \eqref{seq:conn-et_A2}. 
	From the assumption $\Ebar[p^M]\subset \Ebar(\Fk)$, 
	the Galois invariance of the Weil pairing  
	(\cite[Chapter III, Proposition 8.1]{Sil106})
	implies that 
	the determinant of the mod $p^{M}$ representation 
	\[
		G_k\xrightarrow{\rho_M} \Aut(E[p^M])\simeq  GL_2(\Z/p^M)\xrightarrow{\det} (\Z/p^M)^{\times}
	\]
	coincides with the cyclotomic character $\chi_{M}\colon G_k\to \mu_{p^M}$ by fixing 
	a primitive $p^{M}$-th root of unity $\zeta$ which is in $k$. 
	We have $\Ehat[p^{M}] \simeq \mu_{p^{M}}$ as $G_k$-modules and hence $M \le \Nhat$. 
	As we assumed $M = \Mur$, we have $M = \Nhat$.
	The above short exact sequence \eqref{seq:EpM} becomes
	\begin{equation}\label{nonsplitseq}
		0\rightarrow \mu_{p^M}\xrightarrow{\iota_M} E[p^M]\xrightarrow{\pi_M} \Z/p^M\rightarrow 0. 
	\end{equation} 
	Let $\zeta = \zeta_{p^{M}}$ be a fixed primitive $p^{M}$-th root of unity in $k$. 
	Fix a basis $(z,y)$ of $E[p^{M}]$ where 
	$z = \iota_M(\zeta) \in E[p^{M}]$ and $\Ebar[p^{M}]$ is generated by the reduction of $y$. 
	This gives $\Aut(E[p^{M}]) \simeq GL_2(\Z/p^{M})$. 
	If the sequence \eqref{nonsplitseq} splits, then by taking the mod $p^{N+1}$ 
	\[
		\xymatrix{
		G_k \ar[r]^-{\rho_{M}} \ar[rd]_{\rho_{N+1}} & GL_2(\Z/p^{M})\ar[d]^{\mod p^{N+1}}\\ 
		& GL_2(\Z/p^{N+1})
		}
	\]
	the mod $p^{N+1}$ representation $\rho_{N+1}$ becomes semisimple, 
	which contradicts the assumption that the restriction of $\rho_{N+1}$ 
	to the inertia subgroup is irreducible. 
	We conclude that the above short exact sequence \eqref{nonsplitseq} is non-split. 
	
	Applying $G_k$-coinvariance to \eqref{nonsplitseq} we obtain an exact sequence of abelian groups,
	\begin{equation}\label{ses:2.13a}
	(\mu_{p^{M}})_{G_k}\xrightarrow{\iota_M} E[p^{M}]_{G_k}\xrightarrow{\pi_M}(\Z/p^{M})_{G_k}\rightarrow 0.
	\end{equation} 
	\setcounter{claim}{0}
	\begin{claim}\label{clm:1}
		There is an isomorphism 
		$\Im((\mu_{p^{M}})_{G_k}\xrightarrow{\iota_M} E[p^{M}]_{G_k})\simeq\mu_{p^N}\simeq\Z/p^N.$ 
	\end{claim}
	\begin{proof}
		To prove the claim note that the following are true for the sequence \eqref{nonsplitseq}:
		\begin{itemize}
		\item 
		Its corresponding Serre-Tate parameter $u\in \Ok^{\times}/p^M$ is nontrivial.
		\item The $G_k$-action on $E[p^M]$ factors through the cyclic quotient $\Gal(k(u^{1/p^M})/k)$. Let $\sigma\in G_k$ be a lift of a generator of the Galois group $\Gal(k(u^{1/p^M})/k)$.
		\end{itemize}
		For the mod $p^M$ representation $\rho_M\colon G_k \to \Aut(E[p^M])= GL_2(\Z/p^M)$, 
		we have $\rho_M(\sigma) = \begin{pmatrix}
			1 & b \\
			0 & 1
		\end{pmatrix}$ 
		for some $b \in \Z/p^M$. 
		Namely, $\sigma(z,0) = (z,0)$ and $\sigma(0,y) = (bz,y)$.
		Consider the map $p^{M-N} \colon E[p^M]\to E[p^N]$ 
		and $(p^{M-N}z,p^{M-N}y)$ is a basis of $E[p^N]$. 
		The following diagram is commutative 
		\[
		\xymatrix{
		G_k\ar[dr]_{\rho_N}\ar[r]^-{\rho_M} & GL_2(\Z/p^M)  \ar[d]^{\bmod p^N}\\
		& GL_2(\Z/p^N). 
		}
		\]
		Since the action of $G_k$ on $E[p^N]$ is trivial,
		we have  $\begin{pmatrix}
			1 & b \\
			0 & 1
		\end{pmatrix} \equiv \begin{pmatrix}
			1 & 0\\
			0 & 1
		\end{pmatrix} \bmod p^N$ 
		and hence $b \equiv 0 \bmod p^N$. 
		If we suppose $b \equiv 0 \bmod p^{N+1}$, 
		then the action of $G_k$ on $E[p^{N+1}]$ becomes trivial 
		so that $b$ is not divisible by $p^{N+1}$.
	
		Next, we show that $\Ker(\mu_{p^M}\xrightarrow{\iota_M} E[p^M]_{G_k})=\braket{\zeta^{b}}$. 
		Since $(\mu_{p^M})_{G_k}=\mu_{p^M}$, $\zeta^{b}$ is a non-trivial element of $(\mu_{p^M})_{G_k}$. In fact, it is a primitive $p^{M-N}$-th root of unity. We have 
		\[\iota_M(\zeta^{b})=(bz,0)  =\sigma(0,y)-(0,y)=0\in E[p^M]_{G_k}.\] 
		This proves $\braket{\zeta^{b}} \subseteq \ker(\mu_{p^M}\xrightarrow{\iota_M} E[p^M]_{G_k})$. Conversely, let $x \in\ker(\mu_{p^M}\xrightarrow{\iota_M} E[p^M]_{G_k})$. 
		Since the $G_k$-action is cyclic, this means that there exists some $w\in E[p^M]$ such that 
		$\iota_M(x) = \sigma(w)-w$ 
		in $E[p^M]$. 
		Since the $G_k$-action on $\mu_{p^M}$ is trivial, we may assume that  $w=l(0,y)$ for some $l\in\Z/p^M$.  
		Then $\iota_M(x) = l(\sigma(0,y) - (0,y)) = l bz = l \cdot \iota_M(\zeta^b)$. 
		This implies $\ker(\mu_{p^M}\xrightarrow{\iota_M} E[p^M]_{G_k})= \braket{\zeta^{b}}$. We conclude that there is an exact sequence
		\[
		0\rightarrow\mu_{p^M}/\braket{\zeta^{b}} \xrightarrow{\iota_M}E[p^M]_{G_k}\xrightarrow{\pi_M}\Z/p^M\rightarrow 0.
		\] 
		Finally notice that we have an isomorphism 
		$\mu_{p^M}/\braket{\zeta^{b} }\simeq\mu_{p^N}$, since $\braket{\zeta^{b}} \simeq\mu_{p^{M-N}}$, which yields the desired isomorphism 
		$\Im(\iota_M)\simeq\mu_{p^N}\simeq\Z/p^N$.
	\end{proof} 

	\begin{claim}\label{clm:ram}
		The extension $k(E[p^{M}])/k$ is totally ramified.
	\end{claim}
	\begin{proof}
		Let $G$ be the image of the Galois representation $\rho_{M}:G_k\to \Aut(E[p^{M}]) =GL_2(\Z/p^{M})$. 
		We have $G \simeq \Gal(k(E[p^{M}])/k)$. 
		As noted in the proof of \autoref{clm:1}, $G$ is generated by 
		$\begin{pmatrix}
			1 & b \\
			0 & 1
		\end{pmatrix}$
		with $b \equiv 0 \bmod p^{N}$. 
		We have $\#G \le p^{M - N}$. 
		We denote by $I$ the image of the inertia subgroup $I_k = G_{k^{\ur}}$ by $\rho_{M}$ 
		which is isomorphic to the inertia subgroup of $\Gal(k(E[p^{M}])/k)$.
		Since $I \subset G$, it is isomorphic to an additive subgroup of $\Z/p^{M}$, and hence 
		$I$ can be written as 
		\[
		I = \Set{\begin{pmatrix}
		1 & x \\
		0 & 1 	
		\end{pmatrix} |
		x \in p^t(\Z/p^{M})
		}.
		\]
		for some $N \le t\le M$. 
		We consider what happens mod $p^{N+1}$. 
		If we assume $N< t$, then 
		$\begin{pmatrix}1 & x \\
		0 & 1 	
		\end{pmatrix} \in I$ 
		for $x \in p^t(\Z/p^{M})$ is the identity mod $p^{N+1}$, 
		and $x\equiv 0 \bmod p^{N+1}$. 
		The action of $I_k$ on $E[p^{N+1}]$ is trivial. 
		This contradicts to the assumption that $\rho_{N+1}|_{I_k}$ is irreducible. 
		Therefore, $t = N$ and hence $\# I = p^{M-N} = \#G$. 
		The extension $k(E[p^{M}])/k$ is totally ramified.
	\end{proof}

	\begin{claim}\label{clm:3}
	We have an isomorphism $\Im((T_p(\E)^\circ)_{G_k}\xrightarrow{\iota} T_p(E)_{G_k})\simeq\img(\mu_{p^M}\rightarrow E[p^M]_{G_k})$.
	\end{claim}
	\begin{proof} 
		In the appendix (\Cf \autoref{prop:coinv}) we prove isomorphisms 
		$T_p(E)_{G_k} \simeq \plim_n (E[p^n]_{G_k})$ and 
		$T_p(E)_{I_k} \simeq \plim_n (E[p^n]_{I_k})$. 
		We have commutative diagrams 
		\[
		\vcenter {
		\xymatrix{
		(T_p(\E)^\circ)_{I_{k}}\ar[r]^{\iota}\ar[d]^{\simeq} & T_p(E)_{I_{k}}\ar@{->>}[d]\\
		(\mu_{p^{M}})_{I_{k}}\ar[r]^{\iota_M} & E[p^M]_{I_{k}}.
		}}
		\mbox{and}
		\vcenter{
		\xymatrix{
		(T_p(\E)^\circ)_{G_{k}}\ar[r]^{\iota}\ar[d]^{\simeq} & T_p(E)_{G_{k}}\ar@{->>}[d]\\
		(\mu_{p^M})_{G_{k}}\ar[r]^{\iota_M} & E[p^M]_{G_{k}}.
		}
		}
		\]
		Here, the left vertical map in each diagram is bijective by \autoref{lem:Tcoinv} 
		and the assumption $M = \Mur$. 
		Consider the following commutative diagram: 
		\[
		\xymatrix{
		\Im((T_p(\E)^\circ)_{I_k}\xrightarrow{\iota} T_p(E)_{I_k})\ar@{->>}[r]\ar@{->>}[d] & \Im((T_p(\E)^\circ)_{G_{k}}\xrightarrow{\iota} T_p(E)_{G_k})\ar@{->>}[d] \\
		\Im((\mu_{p^M})_{I_k} \xrightarrow{\iota_M} E[p^M]_{I_k}) \ar[r]^-{\simeq} & \Im( (\mu_{p^M})_{G_k} \xrightarrow{\iota_M} E[p^M]_{G_k}). 
		}
		\]
		Here, the bottom horizontal map is bijective 
		because of \autoref{clm:ram}.
		Thus, it is enough to prove the injectivity of the left vertical map in the above diagram. 
		It suffices to show that for every $r>M$ we have an isomorphism $\img((\mu_{p^r})_{I_k}\xrightarrow{\iota_r}E[p^r]_{I_k})\simeq \img((\mu_{p^{M}})_{I_k}\xrightarrow{\iota_M}E[p^M]_{I_k})$. This will follow by \autoref{lem:Tcoinv} and snake lemma. We have a commutative diagram with exact rows and columns
		\[
		\xymatrix{
		& & 0\ar[d] & \\
		(\mu_{p^{r-M}})_{I_k}\ar[r]^{\iota_{r-M}}\ar[d]^{\alpha} & E[p^{r-M}]_{I_k}\ar[r]^{\pi_{r-M}}\ar[d] & \Z/p^{r-M}\ar[r]\ar[d] & 0\\
		(\mu_{p^r})_{I_k}\ar[r]^{\iota_r}\ar[d]^{\beta} & E[p^r]_{I_k}\ar[r]^{\pi_r}\ar[d] & \Z/p^r\ar[r]\ar[d] & 0\\
		(\mu_{p^M})_{I_k}\ar[r]^{\iota_M}\ar[d] & E[p^M]_{I_k}\ar[r]^{\pi_M}\ar[d] & \Z/p^M\ar[r]\ar[d] & 0\\
		0 & 0 & 0 &.
		}\]
		Snake Lemma applied to the rightmost part of the diagram gives an exact sequence  
		\[
		\Ker(\pi_{r-M})\xrightarrow{\alpha}\ker(\pi_r)\xrightarrow{\beta}\ker(\pi_M)\xrightarrow{\delta} \ck(\pi_{r-M}) = 0.
		\] 
		Since $\pi_{r-M}$ is surjective, we get an exact sequence 
		$\ker(\pi_{r-M})\xrightarrow{\alpha}\ker(\pi_r)\xrightarrow{\beta}\ker(\pi_M)\rightarrow 0$. 
		The claim will follow if we show that the map $\ker(\pi_r)\xrightarrow{\beta}\ker(\pi_M)$ is 
		an isomorphism, or equivalently that $\ker(\pi_{r-M})\xrightarrow{\alpha}\ker(\pi_r)$ is 
		the zero map. 
		But this follows by \autoref{lem:Tcoinv}. 
		Namely, the map $(\mu_{p^r})_{I_k}\xrightarrow{\beta}(\mu_{p^M})_{I_k}$ is an isomorphism. 
	\end{proof} 
	From \autoref{bloch}, 
	$\Ker(\dE) \simeq \Im((T_p(\E)^{\circ})_{G_k}\xrightarrow{\iota} T_p(E)_{G_k})$. 
	\autoref{clm:1} and \autoref{clm:3} will complete the proof of the theorem in this case.
	It is clear that if $\overline{E}[p^{\Nhat}]\not\subset\overline{E}(\Fk)$,  the inequality $N\leq\Nhat$ becomes strict. 
\end{proof}

\begin{rem}\label{upperboundachieve} 
	One can use  part (ii) of \autoref{thm:2.9} to construct examples of elliptic curves for 
	which we have $N<\Nhat=M^{\ur}$. In particular, the upper bound of \autoref{thm:main_intro} 
	can be strictly achieved. For example, consider $E$ an elliptic curve over $\Q_p$ 
	with complex multiplication. 
	Let $k_0=\Q_p(\mu_p)$ and for $n\geq 1$ consider the tower of finite extensions 
	$k_n=k_0(\widehat{E}[p^n])$. 
	It follows by \cite[Theorem 2.1.6]{Kaw02} and \cite[IV.6, Theorem 6.1]{Sil106} that 
	for every $n\geq 1$ the extension $k_{n+1}/k_n$ is totally ramified of degree $p$. 
	Thus, there exists some $n\geq 1$ such that 
	$\overline{E}[p^n]\not\subset\overline{E}(\mathbb{F}_{k_n})$. 
	This means that over $k_n$ we have a strict inequality $N<n=\Nhat$. 
	Moreover, notice that  $\Nhat=M^{\ur}$, since $k_{n+1}/k_n$ is totally ramified.  
\end{rem}

\subsubsection*{Construction of the maximal covering} 
We next consider the case when the elliptic curve $E$ is the base change of an elliptic curve over  $\Q$  with  potential complex multiplication. 
Let $E_0$ be an elliptic curve over $\Q$. 
For a field extension $F/\Q$, 
we denote by $\End_F(E_0)$  the ring of endomorphisms on $E_0$ 
which are defined over $F$. 
Assume first, $E_0$ has potential complex multiplication 
by the ring of integers $\mathcal{O}_K$ of an imaginary quadratic field $K$. 
Namely, 
$\End_{\overline{\Q}}(E_0) \simeq \OK$. 
As all endomorphisms on $E_0$ are defined over $K$, 
we also have $\End_{\overline{\Q}}(E_0) = \End_K(E_0) \simeq \OK$.
It follows by \cite[Corollary 5.12]{Rub99} that $K$ has class number one. 
Suppose that the prime number $p$ splits completely in $K$ and $E_0$ has good reduction at $p$.  
We consider the reduction modulo $p$,  
\[
r:\End_K(E_0)\rightarrow\End_{\ol{\F_p}}(\overline{E}_0).
\]
It follows by \cite{Deur41} (see also \cite[13.4, Theorem 12]{Lang87}, \cite[p.\ 2]{Raj69})  that
 there exists a prime element $\eta$ of $\OK$ such that $p=\eta\overline{\eta}$ and  the endomorphism $\eta:E_{0}\rightarrow E_{0}$ of $E_{0}$  reduces to the Frobenius automorphism $\varphi:\overline{E}_{0}\to\overline{E}_{0}$.
 Since $p$ splits completely in $K$, the completion of $K$ at $(\eta)$ is $\Q_p$. 
 Denote by $E = E_0\otimes_{\Q}\Qp$ the base change of $E_0$ to $\Q_p$. We conclude that $E$ has complex multiplication defined over $\Q_p$. That is, $\End_{\Qp}(E)\simeq \mathcal{O}_K$ and $\eta:E\to E$ reduces to the Frobenius. 
   We claim that for every $n\geq 1$, $\ker(\eta^n)=\widehat{E}[p^n]$. Since the reduction of $\eta^n$ is an automorphism, we clearly have $\ker(\eta^n)\subset\widehat{E}$. Moreover, the equality $\eta\overline{\eta}=p$ implies that $\ker(\eta^n)\subset E[p^n]$ from where the claim follows. 
 
 We conclude that if $\ker(\eta^n)=\widehat{E}[p^n]\subset\widehat{E}(k)$, then the isogeny $\eta^n:E\to E$ defines a geometric covering of degree $p^n$ and  is completely ramified over $\Ebar$.   
According to \autoref{thm:2.9}~(i), 
$\eta^{\Nhat}: E\to E$ is the maximal covering corresponding to $\piab(E)^{\geo}_{\ram}$.

\subsection*{Good supersingular reduction}
Next, we consider the elliptic curve $E$  
which has good supersingular reduction. 
The boundary map $\dE:V(E) \to \Ebar(\Fk)$ induces a short exact sequence 
$\Ker(\dE)/p^n \to V(E)/p^n \to  \Ebar(\Fk)/p^n \to 0$. 
As the reduction $\Ebar$  of $E$ satisfies 
$\Ebar[p^n] = 0$ for any $n \ge 1$, 
we have 
$\Ebar(\Fk)/p^n = 0$ and  $\Tor(\Ebar(\Fk),\Z/p^n) \simeq \Ebar(\Fk)[p^n] =  0$ 
so that we obtain 
\begin{equation}
\label{eq:dE}
\Ker(\dE)/p^n \simeq V(E)/p^n. 
\end{equation}
In the following, \emph{we assume} $E[p]\subset E(k)$ 
and will give bounds of $\Ker(\dE)_{\fin}$ (\autoref{thm:main2}). 
By fixing an isomorphism 
$E[p] \simeq (\mu_p)^{\oplus 2}$ of (trivial) $G_k$-modules, 
the Kummer map gives 
\begin{equation}
	\label{eq:KummerE}
	\Ehat(k)/p \inj H^1(k,\Ehat[p]) \simeq H^1(k,\mu_p)^{\oplus 2}\simeq (\kt/p)^{\oplus 2}.
\end{equation}
Its image can be understood by a filtration on $\kt/p$ using the higher unit group $U_k^i = 1 + \mathfrak{m}_k^i$. 
Precisely, because $\Ebar[p] = 0$, 
we have the following decomposition: 
\begin{equation}
	\label{eq:dec}
	E(k)/p \simeq \Ehat(k)/p \simeq \Ubar_k^{p(e_0(k)-t_0(k))} \oplus \Ubar_k^{pt_0(k)}, 
\end{equation}
where 
$\Ubar_k^i := \Im(U_k^i \to k^{\times}/p)$, 
$e_0(k) = e_k/(p-1)$, and $t_0(k) = \max\set{v_k(y) | 0\neq y\in \Ehat[p]}$
(\Cf \cite[Section 3.4]{Gazaki/Hiranouchi2021}).
By this identification \eqref{eq:dec},
we can decompose an element $w$ in $E(k)/p$ as 
$w = (u',u)$ with $u' \in \Ubar_k^{p(e_0(k) -t_0(k))},\ u\in \Ubar_k^{pt_0(k)}$. 
The Galois symbol map associated to $E$ and $\Gm$ (\autoref{def:symbol}) induces 
\[
s_p:(E/p\otimes \Gm/p)(k) \to H^2(k,E[p] \otimes \mu_p) \simeq H^2(k, \mu_p^{\otimes 2})^{\oplus 2} \simeq (\Z/p)^{\oplus 2}.
\]
In fact, this map $s_p$ becomes bijective (\cite[Theorem~4.2]{Hir21}), and since it factors through the surjection $(E/p\otimes \Gm/p)(k)\twoheadrightarrow K(k;E,\Gm)/p$, it follows that this surjection is an isomorphism as well.  
The map above is compatible with the Hilbert symbol map 
$(-,-)_p:\kt/p\otimes \kt/p \to \mu_p \simeq \Z/p$ 
(\cite[Section~XIV.2, Proposition~5]{Ser68})
as the following commutative diagram indicates: 
\begin{equation}\label{diag:H}
	\vcenter{
	\xymatrix{
		E(k)/p \otimes k^{\times}/p  \ar[r]^-{\set{-,-}_{k/k}}\ar[d]^{\simeq} & 
		(E/p \otimes \Gm/p)(k)  \ar[d]^{s_p}_{\simeq} \\ 
	\left(\Ubar_k^{pt_0(k)}\otimes k^{\times}/p \right)\oplus \left( \Ubar_k^{p(e_0(k) - t_0(k))} \otimes k^{\times}/p \right) \ar[r]^-{(-,-)_p^{\oplus 2}}& (\Z/p)^{\oplus 2}.
	}}
\end{equation}
Here,  the top horizontal map is the symbol map  
$w\otimes x \mapsto \set{w,x}_{k/k}$ 
(\Cf \cite[Proof of Proposition~4.6]{Hir21}). 
The above commutative diagram gives the following lemma. 
	
\begin{lem}\label{lem:Hilbert}
	Two elements $\set{ (u'_1,1), x_1}_{k/k}, \set{(1,u_2), x_2}_{k/k}$ 
	generate  $K(k;E,\Gm)/p$ if 
	they satisfy $(u'_1, x_1)_p\neq 0$ and $(u_2, x_2)_p \neq 0$.
\end{lem}

The image of $\Ubar_k^i \otimes \Ubar_k^j$ by the Hilbert symbol is known as follows: 

\begin{lem}[{\cite[Lemma 3.4]{Hir16}}]\label{lem:HilbertIm}
	If $p\nmid i$ or $p\nmid j$, then 
	\[
	\#(\Ubar_k^i,\Ubar_k^j)_p = \begin{cases} 
	p,& \mbox{if $i+j \le pe_0(k)$},\\
	0,& \mbox{otherwise}.
	\end{cases}
	\]
\end{lem}

For $m\ge 1$, put $k_m := k(\mu_{p^m})$.  
Moreover, consider the invariant 
\[
	R=\min\set{r\geq 0 |e_k\leq (p-1)p^r}.
\]   
Using the above observations, we determine generators of $K(k_m;E,\Gm)/p$ for some $m$.

\begin{lem}\label{lem:gen2}
	We assume $E[p]\subset E(k)$ and $M = M^{\ur}$. 
	Then, 
	there exists $M \le m \le M +R$ 
	such that 
	the $K$-group $K(k_m;E,\Gm)/p$ is generated by 
	 elements of the form $\set{a, \zeta_{p^{m}}}_{k_m/k_m}$, 
	where $\zeta_{p^m}$ is a primitive $p^m$-th root of unity. 
\end{lem}
\begin{proof}
	Recalling from \cite[Lemma 3.4]{Gazaki/Hiranouchi2021}, 
	we have $\Ubar_k^i = 1$ for $i>pe_0(k)$ 
	and $\Ubar_k^{i} = \Ubar_k^{i+1}$ for $i$ with $p\mid i$. 
	For some $i\le pe_0(k)$ which is prime to $p$ or $i = pe_0(k)$, 
	we have $\zeta = \zeta_{p^M}\in\Ubar_k^i\setminus\Ubar_k^{i+1}$. 
	From the assumption $M = \Mur$, 
	$k_{M+1} = k(\zeta_{p^{M+1}})/k$ is a totally ramified extension of degree $p$. 
	In the case 	$i = pe_0(k)$, the extension $k_{M+1}/k$ is unramified (\cite[Lemma 2.1.5]{Kaw02}) 
	so we conclude that $i< pe_0(k)$.
	If we have 
	\begin{equation}
	\label{eq:t}
		i \le  \min\set{pt_0(k) ,p(e_0(k) -t_0(k))}, 
	\end{equation}
	then 
	$i + pt_0(k), i + p(e_0(k) - t_0(k)) \le pe_0(k)$. 
	There exist  $u'\in \Ubar_k^{pt_0(k)}$ and $u\in \Ubar_k^{p(e_0(k)-t_0(k))}$ 
	such that $(u',\zeta)_p\neq 0$  and $(u,\zeta)_p \neq 0$ (\autoref{lem:HilbertIm}). 
	Thus, the elements 
	$\set{(u',1), \zeta}_{k/k}$ and $\set{(1,u), \zeta}_{k/k}$ 
	generate $K(k;E,\Gm)/p$ by \autoref{lem:Hilbert}. 
	The assertion holds for $m = M$ and for $k = k_M$. 
	
	Suppose that the above inequality \eqref{eq:t} does not hold. 
	It follows by \autoref{lem:GL} below that if we replace $k$ with $k_{M+1}$, then $\zeta_{p^{M+1}}\in\Ubar_{k_{M+1}}^i\setminus\Ubar_{k_{M+1}}^{i+1}$, while $e_0(k_{M+1})=pe_0(k)$, and $t_0(k_{M+1})=pt_0(k)$. Since $i<pe_0(k)$ and we defined $R$ to be the smallest nonnegative integer such that  $pe_0(k)\leq p^{R}$,
	it follows that 
	there exists $r \le R$ such that  
	over the extension $k_{m} =  k(\mu_{p^{m}})/k$,
	with $m = M+r$,  
	we have 
	\[
		i  \le \min\set{pt_0(k_m),p(e_0(k_m)-t_0(k_m))} 
		= p^r\set{pt_0(k),p(e_0(k)-t_0(k))}. 
	\]
	Applying \autoref{lem:Hilbert} and \autoref{lem:HilbertIm} as above 
	to $k_m$, one can find generators of the form 
	$\set{a, \zeta_{p^m}}_{k_m/k_m}$ as required. 
\end{proof}

\begin{lem}[\Cf {\cite[Lemma~3.23]{Gazaki/Leal2018}} for the case $M\ge 2$]\label{lem:GL}
	We assume $\mu_p\subset k$. 
	Let $x \in \Ubar_k^i\ssm \Ubar_k^{i+1}$, where $0 < i < pe_0(k)$ and $i$ is coprime to $p$. 
	Let $K = k(\sqrt[p]{x})$ and write $\xi = \sqrt[p]{x}$. 
	Then, $\xi \in \Ubar_K^{i}\ssm \Ubar_K^{i+1}$.
\end{lem}
\begin{proof}
	In this proof, we denote by $\ol{x}$ the residue class in $\Ubar_k^i = U_k^i/U_k^i\cap (k^{\times})^p$ represented by 
	the unit $x\in U_k^i$. 
	First, we note that the extension $K/k$ is a totally ramified extension of degree $p$ (\cite[Lemma 2.1.5]{Kaw02}). 
	Thus, $v_K(x-1) = pv_k(x-1) = pi$. 
	Suppose that $\xi = \sqrt[p]{x}$ is in $U_K^j\ssm U_K^{j+1}$ for some $j$ and write $\xi = 1 + u\pi_K^j$ for a unit $u\in \OKt$, where $\pi_K$ is a fixed uniformizer of $K$. 
	From \cite[(5.7)]{FV02}, we calculate the valuation of $\xi^p-1 = x-1$ as follows: 
	\begin{itemize}
		\item If $j > e_0(K) = pe_0(k)$, then $\xi^p\equiv 1+ u' \pi_K^{j+e_K}\bmod \pi_K^{j+e_K+1}$ for some unit $u'\in \OKt$. Thus, 
		\[
		pi = v_K(x-1) = v_K(\xi^p-1) = j+e_K >pe_0(k)+pe_k = p^2e_0(k).
		\]
		This gives $i > pe_0(k)$ and contradicts with the assumption on $i$. 
	\item If $j = e_0(K)$, then $\xi^p \equiv 1+ (u^p + u')\pi_K^{pe_0(K)} \bmod \pi_K^{pe_0(K)+1}$ for some unit $u'\in \OKt$ and hence 
	\[
		pi = v_K(x-1) = v_K(\xi^p-1) \ge pe_0(K) = p^2e_0(k). 
	\]
	Therefore, $i \ge pe_0(k)$, which is again a contradiction. 
	\item If $j < e_0(K)$, then $\xi^p \equiv 1 + u^p\pi_K^{pj}\bmod \pi_K^{pj+1}$. We have 
	\[
		pi = v_K(x-1) = v_K(\xi^p -1) = pj.
	\]
	This implies $i = j$.
	\end{itemize}
	As $\xi \in U_K^i\ssm U_K^{i+1}$, the residue class $\ol{\xi}$ is in $\Ubar_K^i\ssm \Ubar_K^{i+1}$.
\end{proof}

\begin{thm}\label{thm:main2}
	Assume that $E[p]\subset E(k)$. 
	Then, we have surjective homomorphisms 
	\[
		(\Z/p^{M^{\ur} + R})^{\oplus 2} \surj \Ker(\dE)_{\fin} \surj (\Z/p^{N})^{\oplus 2}, 
	\]
	where 
	$R = \min\set{r | e_k < (p-1)p^r}$.
\end{thm}
\begin{proof}
	As we noted in \eqref{eq:dE}, 
	we have $\Ker(\dE)/p^n \simeq V(E)/p^n \simeq K(k;E,\Gm)/p^n$ for any $n\ge 1$. 
	
	\sn 
	\textbf{(Lower bound)}  
	Recalling from \eqref{thm:Som}, we have $\VE \simeq K(k;E,\Gm)$. 
	The lower bound is given by 
	\begin{equation}
	\label{eq:Hir16}
			\Ker(\dE) \surj \Ker(\dE)/p^{N} \simeq K(k;E,\Gm)/p^{N} \simeq (\Z/p^{N})^{\oplus 2}, 
	\end{equation}
	where  the last isomorphism follows from \cite[Remark 4.3]{Hir16}. 
	
	\sn 
	\textbf{(Upper bound)}  
	Since the norm map $K(k';E,\Gm) \to K(k;E,\Gm)$ is surjective 
	for any finite extension $k'/k$ (\cite[Proposition 3.1]{Yam09}), 
	we may assume  $M = M^{\ur}$. 
	In particular, 
	the Kummer extension $k(\mu_{p^{M+1}})/k$ is a totally ramified $p$-extension. 
	Take $m\le M+R$ as in  \autoref{lem:gen2} and put  $k_m = k(\mu_{p^m})$. 
	For each $n\ge m$, 
	we consider the following diagram with exact rows:
	\[
	\xymatrix@C=5mm{
		0 \ar[r]&  K(k_m;E,\Gm)/p \ar@{->>}[d] \ar[r]^{p^n}& K(k_m;E,\Gm)/p^{n+1}\ar@{->>}[d] \ar[r] &  K(k_m;E,\Gm)/p^n\ar[r]\ar@{->>}[d] &  0\\ 
			 0 \ar[r] & K(k;E,\Gm)/p \ar[r]^{p^n}  & K(k;E,\Gm)/p^{n+1} \ar[r] & K(k;E,\Gm)/p^n\ar[r] & 0, 
	  }
	\]
	where the vertical maps are given by norms which are surjective. 
	The far left norm map $K(k_m;E,\Gm)/p \to K(k;E,\Gm)/p$ is bijective 
	because of   
	$K(k_m;E,\Gm)/p\simeq K(k;E,\Gm)/p \simeq \left(\Z/p\right)^{\oplus 2}$ 
	using the assumption $E[p]\subset E(k)$ as in \eqref{eq:Hir16}. 
	It follows by \autoref{lem:gen2} that the map $p^n\colon K(k_m;E,\Gm)/p \to  K(k_m;E,\Gm)/p^{n+1}$ 
	is the 0-map and so is  $p^n \colon K(k;E,\Gm)/p \to  K(k;E,\Gm)/p^{n+1}$. 
	From the above diagram, we have  
	$K(k;E,\Gm)/p^{n+1} \simeq K(k;E,\Gm)/p^n$ for any $n\ge m$. 
	Putting $K = k(E[p^{M+R}])$, 
	there are surjective homomorphisms 
	\[
	(\Z/p^{M+R})^{\oplus 2} \simeq K(K;E,\Gm)/p^{M+R} \surj  K(K;E,\Gm)/p^{m}  \surj K(k;E,\Gm)/p^m. 
	\]
	Here, the last map is induced from the norm map which is surjective. 
	From this, we have 
	\[
	(\Z/p^{M+R})^{\oplus 2}\surj K(k;E,\Gm)/p^n \simeq \Ker(\dE)/p^n
	\]
	for any $n\ge 1$. 
	This implies the existence of a surjective homomorphism 
	$(\Z/p^{M+R})^{\oplus 2} \surj \Ker(\dE)_{\fin}$ as required. 
\end{proof}

\appendix

\section{Profinite Group Homology}
In this appendix, we show the following proposition 
which is used in \cite[(2.21)]{Blo81} and \cite[Section 3]{Som90}: 

\begin{prop}
\label{prop:coinv}
	Let $l$ be a prime, 
	$A$ a semi-abelian variety over a $p$-adic field $k$, and 
	$G$ a closed normal subgroup of $G_k$. 
	Then, we have 
	\[
	T_l(A)_G \simeq \plim_n [(A[l^n])_G].
	\]
\end{prop}

Put $T := T_l(A)$ and $A_n := A[l^n]$.
Using this notation, 
$T = \plim_n A_n$ can be regarded as a profinite $\Z_l[\![G]\!]$-module.
Recall that, for a profinite $\Z_l[\![G]\!]$-module $M$, 
the $m$-th \textbf{homology group} $H_m(G,M)$ of $G$ with coefficients in $M$ is 
given by 
the $m$-th left derived functor of $-\widehat{\otimes}_{\Z_l[\![G]\!]}\Z_l$ 
(\Cf \cite[Section 6.3]{RZ10}). 
The homology group $H_m(G,M)$ can be computed by using the homogeneous bar resolution $L_{\bullet} \surj \Z_l$ as follows: 
\[
H_m(G,M) = H_m(M\widehat{\otimes}_{\Z_l[\![G]\!]}L_\bullet)
\]
(\Cf \cite[Theorem 6.3.1]{RZ10}). 
Each term $L_m$ in $L_{\bullet}$ is a free profinite $\Z_l[\![G]\!]$-module, so that 
we have $\plim(A_n\widehat{\otimes}_{\Z_l[\![G]\!]}L_\bullet) = T\widehat{\otimes}_{\Z_l[\![G]\!]}L_\bullet$ 
and 
\[
H_m(G,T) = H_m(T\widehat{\otimes}_{\Z_l[\![G]\!]}L_\bullet),\ H_m(G,A_n) = 
H_m(A_n\widehat{\otimes}_{\Z_l[\![G]\!]}L_\bullet). 
\]
As $A_n\widehat{\otimes}_{\Z_l[\![G]\!]} L_m = A_n\otimes_{\Z/l^n[\![G]\!]}L_m/l^n$ is finite, 
the tower of chain complexes $\cdots \to A_n\widehat{\otimes}_{\Z_l[\![G]\!]} L_\bullet \to\cdots \to A_1\widehat{\otimes}_{\Z_l[\![G]\!]} L_\bullet$ satisfies the Mittag-Leffler condition.
By \cite[Theorem 3.5.8]{Wei94}, we have an exact sequence for each $m$: 
\[
0 \to {\plim_n}^1 H_{m+1}(G,A_n)\to H_m(G,T) \to \plim_n H_m(G,A_n)\to 0.
\]
In particular, we have 
\[
0 \to {\plim_n}^1 H_1(G,A_n) \to T_G \to \plim_n (A_n)_G\to 0.
\]
Here, $H_1(G,A_n)^{\vee} \simeq H^1(G,A_n^{\vee})$, where 
$\vee$ denotes the Pontrjagin dual. 
Since $A_n^{\vee}$ is finite, 
the action of $G$ on $A_n^{\vee}$ factors through 
a finite quotient $G/K_n$ for some open normal subgroup $K_n\subset G$.  
By the inflation-restriction sequence (\cite[Corollary 7.2.5]{RZ10}, 
we have a short exact sequence
\[
0 \to H^1(G/K_n,A_n^{\vee}) \xrightarrow{\inf} H^1(G,A_n^{\vee}) \xrightarrow{\Res}H^1(K_n,A_n^{\vee}).
\]
As $H^1(K_n,A_n^{\vee}) = \Hom_{\mathrm{cont}}(K_n,A_n^{\vee})$ and 
$H^1(G/K_n,A_n^{\vee})$ are finite abelian groups, so is $H^1(G,A_n^{\vee})$ 
and hence $H_1(G,A_n)$ is finite. 
From this, the tower $\cdots \to H_1(G,A_{n+1}) \to H_1(G,A_{n}) \to \cdots \to H_1(G,A_1)$ satisfies the Mittag-Leffler condition (\Cf \cite[Exercise 3.5.1]{Wei94}). We have $\displaystyle{\plim_n}^1 H_1(G,A_n) = 0$ by \cite[Proposition 3.5.7]{Wei94}).
This gives \autoref{prop:coinv}.

\def\cprime{$'$}
\providecommand{\bysame}{\leavevmode\hbox to3em{\hrulefill}\thinspace}

\providecommand{\href}[2]{#2}


\end{document}